\documentclass[11pt]{article}
\usepackage[cp1251]{inputenc}
\usepackage[english]{babel}
\usepackage{amssymb}
\usepackage{amsmath}
\usepackage{amsfonts}
\usepackage{amsthm}
\usepackage{graphicx}
\usepackage[bookmarks,bookmarksnumbered,bookmarksopen,colorlinks,backref,linkcolor=blue,citecolor=red]{hyperref}%

\newtheorem{theorem}{Theorem}[section]
\newtheorem{definition}{Definition}[section]

\newtheorem{corollary}{Corollary}[section]
\newtheorem{remark}{Remark}[section]
\newtheorem{proposition}{Proposition}[section]

\newcommand{\bl}[1]{{\color{blue}{#1}}}
\newcommand{\rd}[1]{{\color{red}{#1}}}

\newcommand{\gr}[1]{{\color{green}{#1}}}

\begin{document}

\begin{center}
\textbf{\Large Asymptotic approximations of the solution\\[2mm]
  to a boundary-value problem  in a thin aneurysm-type domain}
\end{center}

\medskip

\begin{center}
{\large A. V. Klevtsovskiy \ \& \ T.~A.~Mel'nyk}\\[4mm]
{\small Department of Mathematical Physics, Faculty of Mathematics and Mechanics\\
Taras Shevchenko National University of Kyiv\\
Volodymyrska str. 64,\ 01601 Kyiv,  \ Ukraine\\
 E-mails: avklevtsovskiy@gmail.com, \ \ melnyk@imath.kiev.ua}
\end{center}

\medskip

\begin{abstract}
A nonuniform Neumann boundary-value problem is considered for the Poisson equation  in a thin $3D$ aneurysm-type domain that consists of
 thin curvilinear cylinders that are joined through an aneurysm of diameter $\mathcal{O}(\varepsilon).$
 A rigorous procedure is developed to construct the complete asymptotic expansion for the solution as the parameter $\varepsilon \to 0.$

The asymptotic expansion consists of a regular part that is located inside of each cylinder, a boundary-layer part near  the base of each cylinder, and an inner part discovered in a neighborhood of the aneurysm.
Terms of the inner part of the asymptotics are special solutions of boundary-value problems in an unbounded domain with different outlets at infinity. It turns out that they have polynomial growth at infinity. By matching these parts, we derive the limit problem  $(\varepsilon =0)$ in the corresponding graph and a recurrence procedure to determine all terms of the asymptotic expansion.

Energetic and  uniform pointwise estimates are proved. These estimates allow us to observe the impact of the aneurysm.
\end{abstract}

\bigskip


\medskip

{\bf MOS subject classification:} \  35B40, 74K30, 35C20, 35J05,

\section{Introduction}\label{intro}

In this paper we continue our investigation of boundary-value problems in thin multi-structures with a local
geometric irregularity, which we have begun in \cite{Mel_Klev-2013,Mel_Klev_AA-2016}. Namely,  we modify and generalize our approach for more
complicated structures that consist of thin curvilinear cylinders connected through
a domain of small diameter.

Investigations of various physical and biological processes in channels, junctions and
networks are urgent for numerous fields of natural sciences (see, e.g., \cite{B-P-book,BlanGau03,BlGaMe08,Borysiuk2010,BunCarNaz,C-C-P-2006,CiorPaulin,Ch-Zh-Lukk-Pia-2002,Che-Mel-M2AS-14,MDM-2005,
Gaudiello-Zapp,MN94,M-JMAA-2015,N-book-02,NazPlam,P-book,
P-P-Stokes-1-2015,ZhikovPast} and the references therein). A particular interest is the investigation of the influence of a local geometric heterogeneity in vessels on the blood flow. This is both an aneurysm (a pathological extension of an artery like a bulge) and a stenosis (a pathological restriction of an artery). In \cite{Mardal} the authors classified 12 different aneurysms and proposed a numerical approach for this study.
The aneurysm models have been meshed with 800,000 -- 1,200,000 tetrahedral cells containing three boundary layers.
It was showed that the geometric aneurysm form essentially impacts on the haemodynamics of the blood flow.
However, as was noted by the authors,  {\it the question how to model blood flow with sufficient accuracy  is still open}.

This question was the main motivation for us to begin the study of boundary-value problems in domains of such type and
 to detect the influence of their local geometric irregularity on properties of solutions.
It is clear that such domains are  prototypes of many other biological and engineering structures, but  we prefer to call
them {\it aneurysm-type domains} as comprehensive and concise.

In this paper we use the asymptotic approach for boundary-value problems in thin domains.
The idea is as follows. Let us consider a boundary-value problem in a neighborhood of an aneurysm. After the nondimensionalization, we get a parameter $\varepsilon$  characterizing  thickness of the domain. In many cases, e.g. for brain aneurysms, $\varepsilon$ is a small parameter. 
Thus, it is natural to study the behaviour of this problem as $\varepsilon$ tends to zero. As we can see from Fig.~\ref{Fig-1}, 
 the thin aneurysm-type domain is transformed into a graph and the aneurysm is transformed into the origin if $\varepsilon \to 0.$

\begin{figure}[htbp]
\begin{center}
\includegraphics[width=7cm]{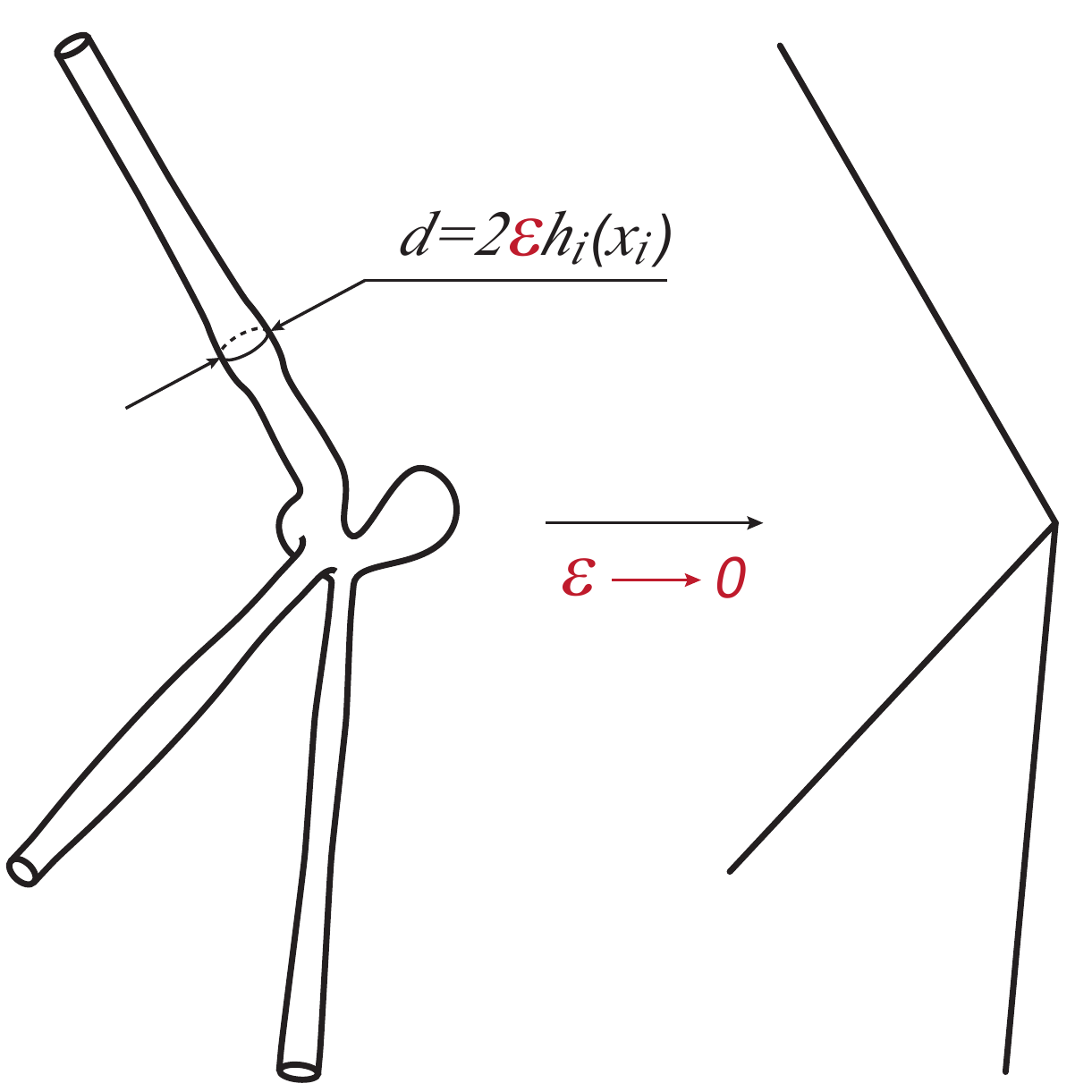}
\caption{Transformation of a thin aneurysm-type domain into a graph}\label{Fig-1}
\end{center}
\end{figure}

So, the aim is to find the corresponding limit problem and detect the impact of the aneurysm.
Obviously, this limit problem in the graph will be simpler, since it is one-dimensional problem.
Then we can either analytically solve the limit problem or apply numerical methods.

There are several approaches to construct asymptotic approximations for solutions to boundary-value problems in thin rod structures.
The method of the partial asymptotic domain decomposition (MPADD), proposed in \cite{Pan-decom-1998},  was applied in
the book \cite[Chapter 4]{P-book} to the following problem in a finite thin rod structure $B_{\varepsilon}$:
$$
\Delta u_\varepsilon = \left\{
             \begin{array}{ll}
               f^e(\widetilde{x}_1), & \hbox{in} \ S_0,
\\
               0, & \hbox{in} \ \Pi_{x_0} := B_{\varepsilon}\setminus S_0,
             \end{array}
           \right.
$$
$$
u_\varepsilon|_{\partial_1 B_{\varepsilon}} = 0, \quad \frac{\partial u_\varepsilon}{\partial n}|_{\partial_2 B_{\varepsilon}} = 0.
$$
Here $S_0$ is the union of sections of the rod structure $B_{\varepsilon},$ $\Pi_{x_0}$ is the connected component of  $B_{\varepsilon}\setminus S_0$ containing the node $x_0.$  The main idea of this method is to reduce the problem to a simplified form on $S_0,$ where the regular asymptotics of the solution is located. The initial formulation is kept on a small neighbourhood of the domain $\Pi_{x_0},$ where the asymptotic behavior is singular.
 Then these two models are coupled by some special interface conditions that are derived from
 some projection procedure. For this,  the author proposed a method of  redistribution of constants to have the boundary-layer solutions exponentially decaying at infinity. As a result, for the leading term of the asymptotic expansion the following estimate was obtained:
\begin{equation}\label{P}
\frac{\|u_\varepsilon - v\|_{L^2(B_{\varepsilon})}}{\sqrt{meas (B_{\varepsilon})}} = \mathcal{O}(\sqrt{\varepsilon})\quad \text{as} \quad \varepsilon \to 0.
\end{equation}
Then this method  has been applied to constructions of  asymptotic expansions both for the solution of
the wave equation on a thin rod structure \cite{P-wave-2015},  for
the solution of non-steady Navier-Stokes equations with uniform boundary conditions in a thin tube structure $B_\varepsilon$ \cite{P-P-Stokes-1-2015,P-P-Stokes-2-2015,Car-1},  for the uniform Dirichlet boundary value problem for the biharmonic
equation in a thin T-like shaped plane domain \cite{GPP-2016} and for other problems \cite{C-C-P-2006,Car-2}. Thus, the method (MPADD) is used in the case of the uniform boundary conditions on the lateral rectilinear surfaces of thin rods (cylinders) and if the right-hand sides depend only on the longitudinal variable in the direction of the corresponding rod and they are constant in some neighbourhoods of the nodes and vertices.

We see that the main difficulty in such problems is the identification of the behaviour of solutions in neighbourhoods of the nodes (aneurysms).
In \cite{N-S-02,N-S-04} the authors made the following assumptions:
\begin{itemize}
  \item
  the first terms of the volume force $f$ and surface load $g$ on the rods satisfy special orthogonality conditions (see $(3.5)_1$ and $(3.6)$  in \cite{N-S-04}) and the second term of the volume force $f$ has an identified form and depends only on the longitudinal variable,
  \item
  similar orthogonality conditions for the right-hand sides on the knots are satisfied (see $(3.41)$) and the second term is a piecewise constant vector-function (see $(3.42)$),
\end{itemize}
to overcome this difficulty and to construct the leading terms of the elastic-field asymptotics
for solutions of the equations of anisotropic elasticity on junctions of thin beams ($2D$ and $3D$ cases).
Due to these assumptions the displacement field at each knot can be approximated by a rigid displacement. As a result, the approximation
does not contain boundary-layer terms, i.e. {\it the asymptotic expansion is not complete a priori} \cite[Remark 3.1]{N-S-04}.

Using the method of two-scale expansions, the complete asymptotic expansion in powers $\varepsilon$  and $\mu$  for a solution of a linear partial deferential equation in the simplest $s$-dimensional rectangular periodic carcass  was constructed in \cite{B-P-book} (here $\varepsilon$ is the period of the carcass and  $(\varepsilon\mu)^{s-1}$ is the area of the cross-section of beams).

\subsection{Novelty and Methods}
A new feature of the present paper in comparison with the papers mentioned above
is construction and justification of the complete asymptotic expansion for the solution to
a nonuniform Neumann boundary-value problem for the Poisson equation  in a thin $3D$ aneurysm-type domain
and the proof both energetic and pointwise uniform estimates for the difference between the solution of the starting problem $(\varepsilon >0)$ and the solution of the corresponding limit problem  $(\varepsilon =0)$ without any orthogonality conditions for the right-hand side in the equation and for the right-hand sides in the Neumann boundary conditions. In addition,  the right-hand sides can depend both on longitudinal and transversal variables and  the thin  cylinders can be curvilinear.

To construct the asymptotic expansion in the whole domain, we use the  method of matching asymptotic expansions (see  \cite{I}) with special cut-off functions. The asymptotic expansion consists of three parts, namely, the regular part of the asymptotics  located inside of each thin cylinder, the boundary-layer part  near the base of each thin cylinder, and the inner part of the asymptotics discovered in a neighborhood of the aneurysm.

The terms of the inner part of the asymptotics are special solutions of boundary-value problems in an unbounded domain with different outlets at infinity. It turns out they have polynomial growth at infinity. Matching these parts, we derive the limit problem  $(\varepsilon =0)$ in the corresponding graph and the recurrence procedure to determine all terms of the asymptotic expansion.

We proved energetic estimates that allow us to identify more precisely the impact  of the aneurysm on some properties of the whole structure. One of the main results obtained in this paper is the energetic estimate (see Corollary~\ref{corollary1})
\begin{equation}\label{our}
\frac{\|u_\varepsilon - U^{(0)}\|_{H^1(\Omega_{\varepsilon})}}{\sqrt{meas (\Omega_{\varepsilon})}} = \mathcal{O}(\varepsilon^{\frac{\alpha}{2}})\quad \text{as} \quad \varepsilon \to 0
\end{equation}
in the Sobolev space $H^1(\Omega_{\varepsilon})$ instead of $L^2$-space in \cite[Chapter 4]{P-book} (see (\ref{P})). Here
 $\alpha$ is a fixed number from the interval $(\frac23 , 1).$
In addition,  pointwise uniform estimates are deduced in the case of rectilinear cylinders.

It should be stressed that the error estimates and convergence rate are very important both for  justification of  adequacy of one- or two-dimensional models that aim at  description of actual three-dimensional thin bodies and for the study of boundary effects and effects of local (internal) inhomogeneities in applied problems. Pointwise estimates are of particular importance for engineering practice, since large values of tearing stresses in a small region at first cause local material damage and then lead to destruction of the  whole construction.

Thus, our approach makes it possible to take into account various factors (e.g. variable thickness of thin curvilinear cylinders,
 inhomogeneous boundary conditions, geometric characteristics of aneurysms, etc.) in statements of boundary-value problems on graphs.
In addition, the transition from two-dimensional to three-dimensional problems, the variable thickness of thin cylinders and their arbitrary number,
special behaviour of the inner terms, and the matching procedure are the major differences and difficulties that were overcome in contrast to our previous paper \cite{Mel_Klev_AA-2016}.

 Also this approach has been applied to nonlinear monotone boundary-value problems with nonlinear boundary conditions in thin aneurysm-type domains to construct the leading terms of the asymptotics and these results were announced at the conference \cite{Kle}. It should be mentioned here that in \cite{Mel_Klev-2013} we studied a boundary-value problem in the union of thin rectangles without any local geometric irregularity. In this case the inner part of the asymptotics is absent and we have to equate simply the regular parts to obtain transmission conditions in the respective limit problem.

\subsection{Structure of the paper}

In section \ref{statement}, we describe the domain $\Omega_\varepsilon$ and the statement of the problem.
The formal asymptotic expansion for the solution to the problem (\ref{probl}) is constructed in  section~\ref{Formal asymptotics}.  In section~\ref{justification}, we justify the asymptotics (Theorem~\ref{mainTheorem}) and prove asymptotic estimates for the leading terms of the
asymptotics (Corollaries~\ref{corollary1}, \ref{corollary2} and \ref{corollary3}). In the Conclusions,  we analyze results obtained in this paper and discuss possible generalizations.

\section{Statement of the problem}\label{statement}

The model thin aneurysm-type domain  $\Omega_\varepsilon$  consists of three thin curvilinear cylinders
$$
\Omega_\varepsilon^{(i)} =
  \Bigg\{
  x=(x_1, x_2, x_3)\in\Bbb{R}^3: \
  \varepsilon \ell<x_i<1, \ \
  \sum\limits_{j=1}^3 (1-\delta_{ij})x_j^2<\varepsilon^2 h_i^2(x_i)
  \Bigg\}, \ i=1,2,3,
$$
that are joined through a domain $\Omega_\varepsilon^{(0)}$ (referred in the sequel "aneurysm").
Here $\varepsilon$ is a small parameter; $\ell\in(0, \frac13);$  the positive functions $\{h_i\}_{i=1}^3$ belong to the space $C^1 ([0, 1])$
and they are equal to some constants in neighborhoods at the points $x=0$ and $x_i=1, \ i=1,2,3$; the symbol $\delta_{ij}$ is the Kroneker delta, i.e.,
$\delta_{ii} = 1$ and $\delta_{ij} = 0$ if $i \neq j.$

The aneurysm $\Omega_\varepsilon^{(0)}$ (see Fig.~\ref{Fig-2})
is formed by the homothetic transformation with coefficient $\varepsilon$ from a bounded domain
$\Xi^{(0)}\in \Bbb R^3$,  i.e.,
$
\Omega_\varepsilon^{(0)} = \varepsilon\, \Xi^{(0)}.
$
In addition, we assume that its boundary contains the disks
$$
\Upsilon_\varepsilon^{(i)} (\varepsilon\ell)=\Bigg\{ x\in\Bbb{R}^3 \, : \ x_i=\varepsilon \ell, \quad \sum\limits_{j=1}^3 (1-\delta_{ij})x_j^2<\varepsilon^2 h_i^2(\varepsilon \ell) \Bigg\}, \quad i=1,2,3,
$$
and denote
$
\Gamma_\varepsilon^{(0)} := \partial\Omega_\varepsilon^{(0)} \backslash \left\{\overline{\Upsilon_\varepsilon^{(1)} (\varepsilon \ell)}
\cup \overline{\Upsilon_\varepsilon^{(2)} (\varepsilon \ell)} \cup \overline{\Upsilon_\varepsilon^{(3)} (\varepsilon \ell)}\right\}.
$

\begin{figure}[htbp]
\begin{center}
\includegraphics[width=7cm]{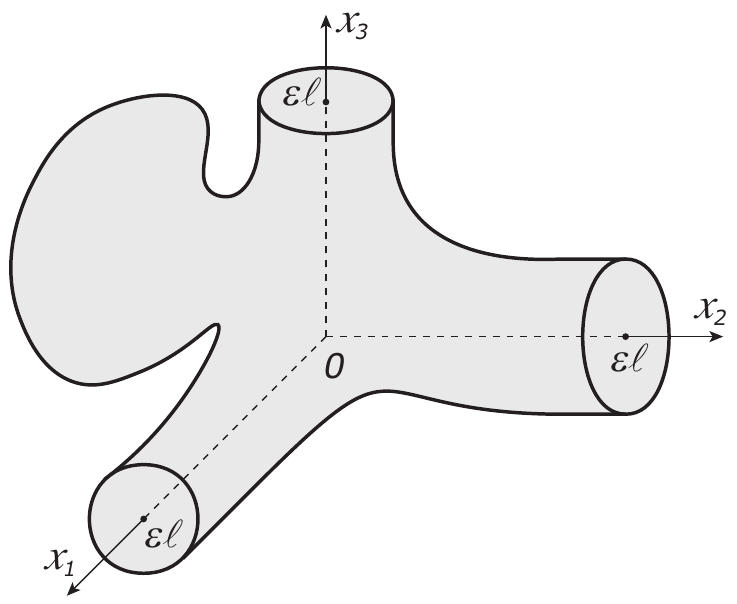}
\caption{The aneurysm $\Omega_\varepsilon^{(0)}$}\label{Fig-2}
\end{center}
\end{figure}

Thus the model thin aneurysm-type domain  $\Omega_\varepsilon$  (see Fig.~3) 
is   the interior of the union
$
\bigcup_{k=0}^{3}\overline{\Omega_\varepsilon^{(k)}}
$
and we assume that it has the Lipschitz boundary.

\begin{remark}
We can consider more general thin aneurysm-type domains with arbitrary orientation of thin cylinders (their number can be also arbitrary).
 But to avoid technical and huge calculations and to demonstrate the main steps of the proposed asymptotic approach
 we consider a such kind of the thin aneurysm-type domain, when the cylinders are placed on the coordinate axes.
\end{remark}

\begin{figure}[htbp]
\begin{center}
\includegraphics[width=7.5cm]{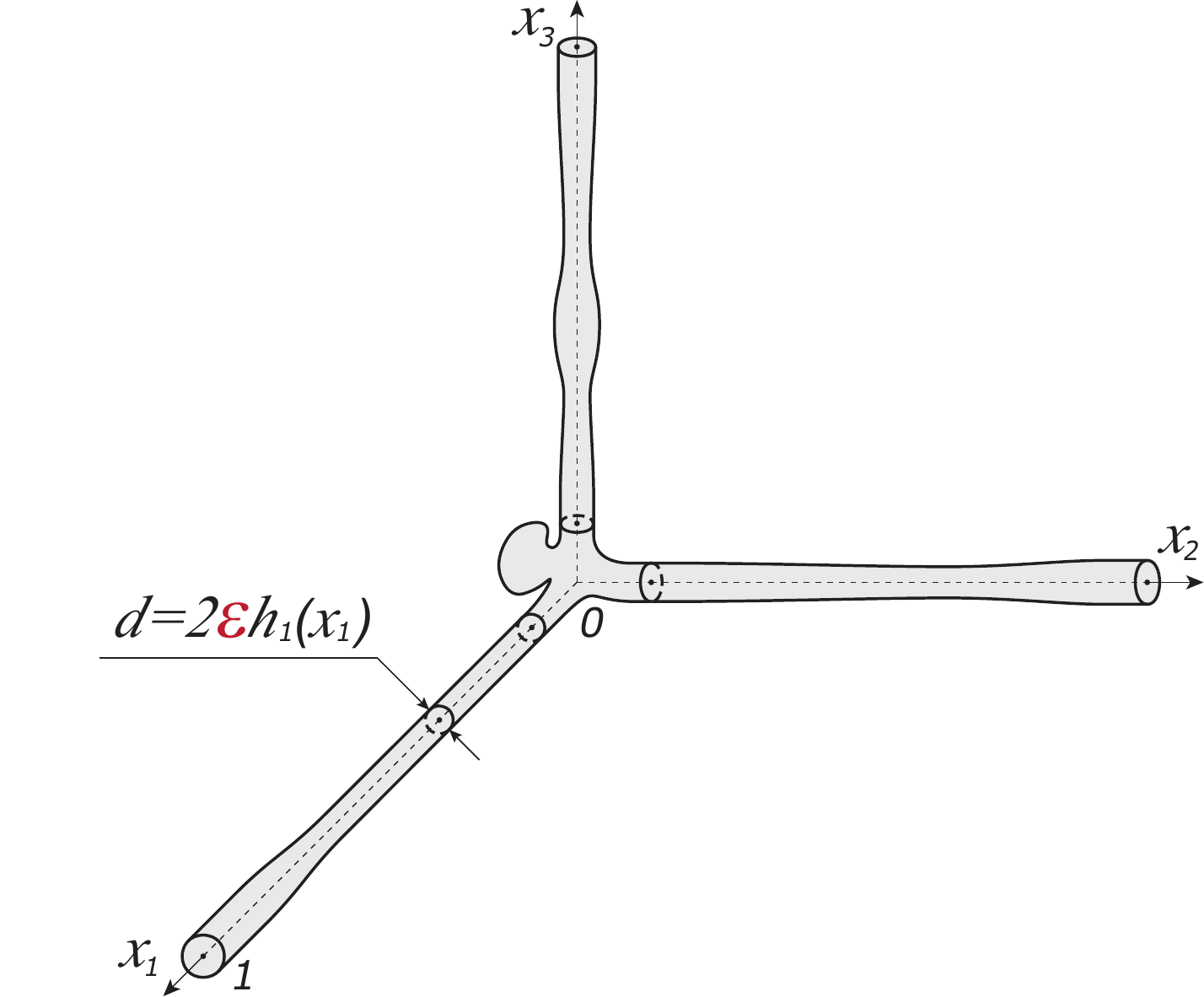}
\caption{The model thin aneurysm-type domain $\Omega_\varepsilon$ }\label{Fig-3}
\end{center}
\end{figure}

In $\Omega_\varepsilon,$ we consider the following mixed boundary-value problem:
\begin{equation}\label{probl}
\left\{\begin{array}{rclll}
  -\Delta{u_\varepsilon}(x) & = & f(x), &
   x\in\Omega_\varepsilon, &
\\
  \partial_{\boldsymbol{\nu}}{u_\varepsilon}(x) & = & 0, &
   x\in{\Gamma_\varepsilon^{(0)}}, &
\\
  -\partial_{\boldsymbol{\nu}}{u_\varepsilon}(x) & = & {\varphi_{\varepsilon}(x)}, &
   x\in{\Gamma_\varepsilon^{(i)}}, & i=1,2,3,
\\
  u_\varepsilon(x) & = & 0, &
   x\in{\Upsilon_{\varepsilon}^{(i)} (1)}, & i=1,2,3,
\end{array}\right.
\end{equation}
where
$
{\Gamma_\varepsilon^{(i)}} = \partial\Omega_\varepsilon^{(i)} \cap \{ x\in\Bbb{R}^3 \ : \ \varepsilon \ell<x_i<1 \},
$
$\partial_{\boldsymbol{\nu}}$ is the outward normal derivative, \\ the given functions $f$ and $\varphi_{\varepsilon}$ are
smooth  and
$$
\varphi_{\varepsilon} (x) = \ \varepsilon \varphi^{(i)} \left(x_i, \dfrac{\overline{x}_i}{\varepsilon} \right), \quad x\in\Gamma_\varepsilon^{(i)}, \quad i=1,2,3,
$$
where
$$
\overline{x}_i =
\left\{\begin{array}{lr}
(x_2, x_3), & i=1, \\
(x_1, x_3), & i=2, \\
(x_1, x_2), & i=3.
\end{array}\right.
$$
It follows from the theory of linear boundary-value problems that, for any fixed value of $\varepsilon,$ the problem~(\ref{probl})
possesses a unique weak solution $u_\varepsilon$ from the Sobolev space $H^1(\Omega_\varepsilon)$ such that its traces on the
ends $\Upsilon_{\varepsilon}^{(i)}(1),  i=1,2,3,$ of the domain $\Omega_\varepsilon$ are equal to zero and the solution satisfies the integral identity
\begin{equation}\label{int-identity}
     \int_{\Omega_\varepsilon} \nabla u_\varepsilon \cdot \nabla \psi \, dx
 =   \int_{\Omega_\varepsilon} f \,\psi \, dx
 -   \sum_{i=1}^3 \int_{\Gamma_\varepsilon^{(i)}} \varphi_\varepsilon \, \psi\, d\sigma_x
\end{equation}
for any function $\psi\in H^1(\Omega_\varepsilon)$ such that $\psi|_{x_i=1}=0, \ i=1,2,3.$

{\sf
The aim of the present paper is to
\begin{itemize}
  \item
  develop a procedure to construct the complete asymptotic expansion for the solution to the problem  (\ref{probl}) as the small parameter $\varepsilon \to 0;$
\item
  justify this procedure and prove the corresponding asymptotic estimates;
\item
   derive the corresponding limit problem $(\varepsilon =0)$;
\item
  prove energetic and uniform pointwise estimates for the difference between the solution of the problem~ (\ref{probl})  and the solution of
   the limit problem,   from which the influence of the aneurysm will be observed.
 \end{itemize}
}

\section{Formal asymptotic expansions}\label{Formal asymptotics}

We propose the following asymptotic ansatzes for the solution to the problem ~(\ref{probl}) :
\begin{enumerate}
  \item[1)]
 the regular part of the asymptotics
  \begin{equation}\label{regul}
\bl{
u_\infty^{(i)} := \omega_0^{(i)} (x_i) + \varepsilon \, \omega_1^{(i)} (x_i)
 +  \sum\limits_{k=2}^{+\infty} \varepsilon^{k}
    \left(
    u_k^{(i)} \left( x_i, \dfrac{\overline{x}_i}{\varepsilon} \right)
 +  \omega_k^{(i)} (x_i)
    \right)
}
\end{equation}
is located inside of each thin cylinder $\Omega^{(i)}_\varepsilon$ (see Fig.~\ref{Fig-4})
and their terms  depend both on the corresponding longitudinal variable $x_i$ and so-called "quick variables" $\dfrac{\overline{x}_i}{\varepsilon} \ (i=1,2,3);$
  \item[2)]
the boundary-layer part of the asymptotics
\begin{equation}\label{prim+}
\gr{\Pi_\infty^{(i)}
 := \sum\limits_{k=0}^{+\infty}\varepsilon^{k}\Pi_k^{(i)}
    \left(\frac{1-x_i}{\varepsilon},\frac{\overline{x}_i}{\varepsilon}\right)} \quad (i=1,2,3)
\end{equation}
is located in a neighborhood of the base $\Upsilon_{\varepsilon}^{(i)} (1)$ of each thin cylinder $\Omega^{(i)}_\varepsilon;$
  \item[3)]
the inner part of the asymptotics
  \begin{equation}\label{junc}
\rd{N_\infty=\sum\limits_{k=0}^{+\infty}\varepsilon^k N_k\left(\frac{x}{\varepsilon}\right)}
\end{equation}
is located in a neighborhood of the aneurysm $\Omega^{(0)}_\varepsilon$.
\end{enumerate}
\begin{figure}[htbp]
\begin{center}
\includegraphics[width=9cm]{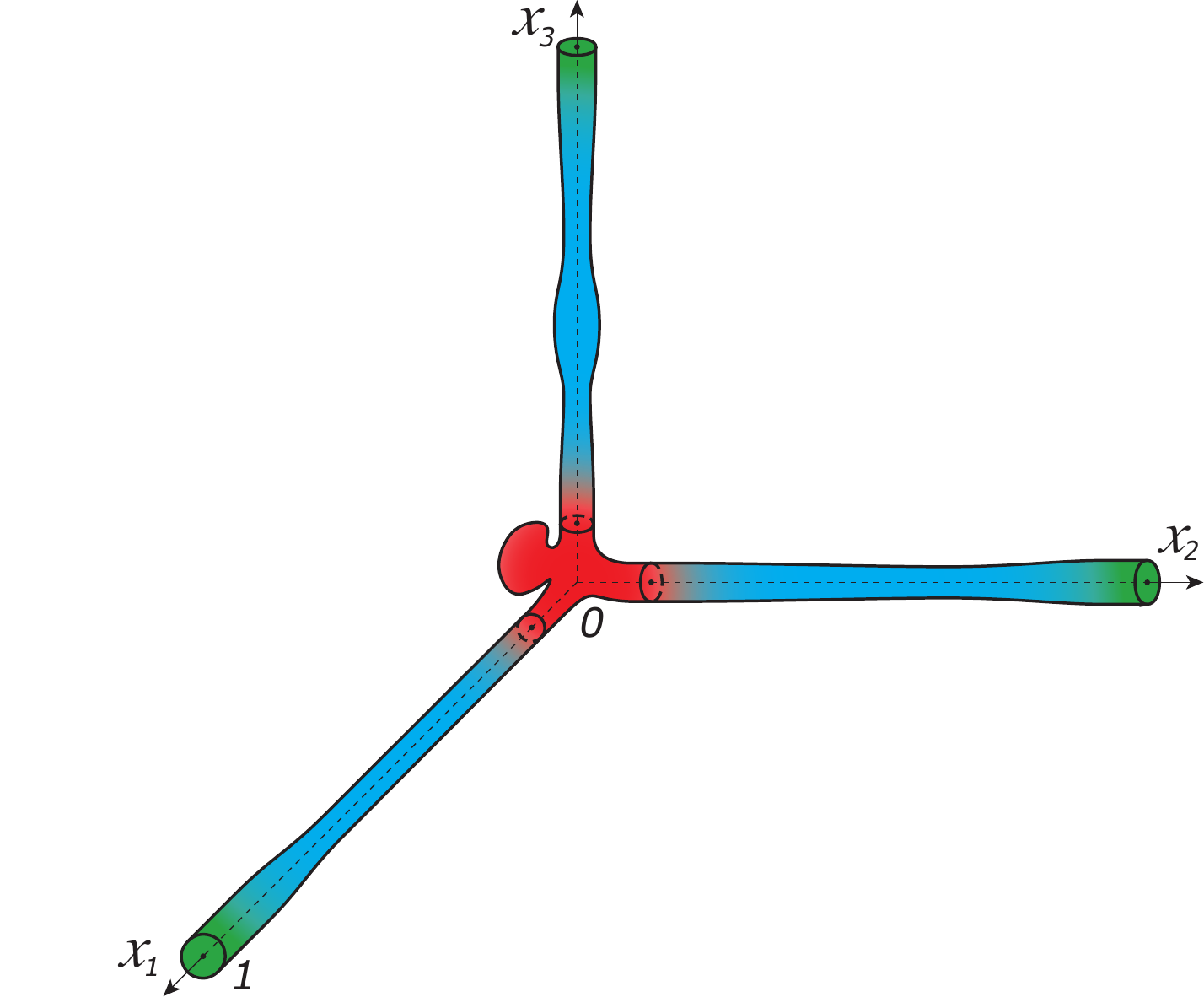}
\caption{Location of different parts of the asymptotics}\label{Fig-4}
\end{center}
\end{figure}

\subsection{Regular part of the asymptotics}
Formally substituting the series (\ref{regul}) into the differential equation  of the problem (\ref{probl}) and expanding function $f$ in the Taylor series at the point $\overline{x}_i=(0,0),$ we obtain
$$
 -  \sum\limits_{k=2}^{+\infty} \varepsilon^{k}
    \frac{d^2{u}_{k}^{(i)}}{d{x_i}^2} (x_i,\overline{\xi}_i)
 -  \sum\limits_{k=0}^{+\infty} \varepsilon^{k}
    \Delta_{\overline{\xi}_i}{u}_{k+2}^{(i)} (x_i,\overline{\xi}_i)
 -  \sum\limits_{k=0}^{+\infty} \varepsilon^{k}
    \frac{d^2\omega_k^{(i)}}{d{x_i}^2} (x_i)
  \approx
    \sum\limits_{k=0}^{+\infty} \varepsilon^{k} f_k^{(i)}(x_i,\overline{\xi}_i),
$$
where $\xi_i = \dfrac{{x}_i}{\varepsilon}, \ \overline{\xi}_i = \dfrac{\overline{x}_i}{\varepsilon}$ and
\begin{equation}\label{f_k}
f_k^{(i)}(x_i,\overline{\xi}_i)
 :=  \dfrac{1}{k!}
    \left(
    \sum\limits_{j=1}^3 (1-\delta_{ij}) \, \xi_j \, \frac{\partial }{\partial{x_j}}
    \right)^k f(x)|_{\overline{x}_i=(0,0)},
\quad k\in \Bbb N_0;
\end{equation}
e.g. at $i=1$ the last symbol means as follows $f_0^{(1)}(x_1, \xi_2, \xi_3)= f(x_1,0,0)$ and
\begin{equation*}
f_k^{(1)}(x_1, \xi_2, \xi_3)
 =  \dfrac{1}{k!}
    \left(
    \sum\limits_{n=0}^k \left(
                          \begin{array}{c}
                            k \\
                            n \\
                          \end{array}
                        \right)
     \xi_2^n \, \frac{\partial^n f}{\partial{x_2^n}}\big(x_1,0,0\big) \ \xi_3^{k-n} \, \frac{\partial^{k-n} f}{\partial{x_3^{k-n}}}\big(x_1,0,0\big)
    \right),
\quad k\in \Bbb N.
\end{equation*}

Then, taking into account the view of  the outer normal to $\Gamma^{(i)}_\varepsilon$
\begin{eqnarray*}
{\boldsymbol{\nu}}^{(i)} (x_i, \ \overline{\xi}_i) &=& \dfrac{1}{\sqrt{1 + \varepsilon^2 |h_i^\prime (x_i)|^2 \,}}
    \big( -\varepsilon h_i^\prime (x_i), \ \overline{\nu}_i (\overline{\xi}_i) \big)
\\
&=&
\left\{\begin{array}{lr}
    \dfrac{ \big(
   -\varepsilon h_1^\prime (x_1), \
    {\nu}_2^{(1)} (\overline{\xi}_1), \
    {\nu}_3^{(1)} (\overline{\xi}_1)
    \big) }{\sqrt{1 + \varepsilon^2 |h_1^\prime (x_1)|^2 \,}},
    \ & i=1,
\\
    \dfrac{ \big(
    {\nu}_1^{(2)} (\overline{\xi}_2), \
   -\varepsilon h_2^\prime (x_2), \
    {\nu}_3^{(2)} (\overline{\xi}_2)
    \big) }{\sqrt{1 + \varepsilon^2 |h_2^\prime (x_2)|^2 \,}},
    \ & i=2,
\\
    \dfrac{ \big(
    {\nu}_1^{(3)} (\overline{\xi}_3), \
    {\nu}_2^{(3)} (\overline{\xi}_3), \
   -\varepsilon h_3^\prime (x_3)
    \big) }{\sqrt{1 + \varepsilon^2 |h_3^\prime (x_3)|^2 \,}},
    \ & i=3,
\end{array}\right.
\end{eqnarray*}
where $\overline{\nu}_i (\frac{\overline{x}_i}{\varepsilon})$ is the outward normal for the disk ${\Upsilon_{\varepsilon}^{(i)} (x_i)}$
for each value of $x_i, \ i=1,2,3,$ we put the series (\ref{regul}) into  the third  relation of the problem (\ref{probl}) and derive
$$
    h_i^\prime(x_i)
    \sum\limits_{k=3}^{+\infty}\varepsilon^{k}
    \frac{d{u}_{k-1}^{(i)}}{dx_i} (x_i,\overline{\xi}_i)
 -  \sum\limits_{k=1}^{+\infty}\varepsilon^{k}
    \partial_{\overline{\nu}_i (\overline{\xi}_i)}{u}_{k+1}^{(i)} (x_i,\overline{\xi}_i)
 +  h_i^\prime(x_i)
    \sum\limits_{k=1}^{+\infty}\varepsilon^{k}
    \frac{d\omega_{k-1}^{(i)}}{d{x_i}} (x_i)
$$
$$
 \approx
    \varepsilon \sqrt{1+\varepsilon^2 |h_i^\prime(x_i)|^2 \ }
 \cdot
    \varphi^{(i)}(x_i,\overline{\xi}_i)
 =  \sum\limits_{k=0}^{+\infty} \varepsilon^{2k+1}
    \dfrac{(-1)^k (2k)!}{(1-2k)(k!)^2 4^k}
    |h_i^\prime(x_i)|^{2k} \, \varphi^{(i)}(x_i,\overline{\xi}_i).
$$

Equating the coefficients of the same powers of $\varepsilon$, we deduce recurrent relations of the boundary-value problems
for the  determination of the expansion coefficients in (\ref{regul}). Let us consider the problem for  $u_2^{(i)}:$
\begin{equation}\label{regul_probl_2}
\left\{\begin{array}{rcll}
-\Delta_{\overline{\xi}_i}{u}_{2}^{(i)} (x_i,\overline{\xi}_i)
 & = &
\dfrac{d^{\,2}\omega_0^{(i)}}{d{x_i}^2} (x_i) + f_0^{(i)} (x_i,\overline{\xi}_i),
 & \ \ \overline{\xi}_i\in\Upsilon_i (x_i),
\\[2mm]
-\partial_{\nu_{\overline{\xi}_i}}{u}_{2}^{(i)}(x_i,\overline{\xi}_i)
 & = &
- \ h_i^\prime (x_i)\dfrac{d\omega_0^{(i)}}{d{x_i}} (x_i) + \varphi^{(i)} (x_i,\overline{\xi}_i),
 & \ \ \overline{\xi}_i\in\partial\Upsilon_i(x_i),
\\[2mm]
\langle u_2^{(i)} (x_i,\cdot) \rangle_{\Upsilon_i (x_i)}
 & = &
0.
 &
\end{array}\right.
\end{equation}
Here the variable $x_i$ is regarded as a parameter from the interval $ \ I_\varepsilon^{(i)} =\{x: \ x_i\in  (\varepsilon \ell, 1), \ \overline{x_i}=(0,0)\},$
$\Upsilon_i(x_i)=\{ \overline{\xi}_i\in\Bbb{R}^2 \ : \ |\overline{\xi}_i|< h_i(x_i) \}, $ \
$\langle u(x_i,\cdot) \rangle_{\Upsilon_i(x_i)} :=  \int_{\Upsilon_i(x_i)}u (x_i,\overline{\xi}_i)d{\overline{\xi}_i},$ $\ i=1,2,3.$

For each value of $i\in \{1, 2, 3\},$ the problem (\ref{regul_probl_2}) is the inhomogeneous Neumann problem for the Poisson equation in the disk $\Upsilon_i(x_i)$ with respect to the variable ${\overline{\xi}_i}\in\Upsilon_i(x_i).$  Writing down
the necessary and sufficient conditions for the solvability of problem (\ref{regul_probl_2}), we get the following differential equation for the function $\omega_0^{(i)}:$
\begin{equation}\label{omega_probl_2}
 -  \pi \dfrac{d}{d{x_i}}\left(h_i^2(x_i)\frac{d\omega_0^{(i)}}{d{x_i}}(x_i)\right)
 =  \int\limits_{\Upsilon_i(x_i)}f_0^{(i)}(x_i,\overline{\xi}_i) \, d{\overline{\xi}_i}
 -  \int\limits_{\partial\Upsilon_i(x_i)}
    \varphi^{(i)}(x_i,\overline{\xi}_i) \, dl_{\overline{\xi}_i},
\quad x_i\in I_\varepsilon^{(i)}.
\end{equation}
Let $\omega_0^{(i)}$ be a solution of the differential equation (\ref{omega_probl_2}) (boundary conditions for this differential equation will be determined later). Then the solution of problem (\ref{regul_probl_2}) exist and the third relation in (\ref{regul_probl_2}) supplies the uniqueness of solution.

For determination of the coefficients $u_3^{(i)}, \  i=1,2,3,$ we obtain the following problems:
\begin{equation}\label{regul_probl_3}
\left\{\begin{array}{rcll}
-\Delta_{\overline{\xi}_i}{u}_{3}^{(i)}(x_i,\overline{\xi}_i)
 & = &
\dfrac{d^{\,2}\omega_1^{(i)}}{d{x_i}^2}(x_i) + f_1^{(i)}(x_i,\overline{\xi}_i),
 &\quad
\overline{\xi}_i\in\Upsilon_i(x_i),
 \\[2mm]
-\partial_{\nu_{\overline{\xi}_i}}{u}_{3}^{(i)}(x_i,\overline{\xi}_i)
 & = &
- \ h_i^\prime(x_i)\dfrac{d\omega_1^{(i)}}{d{x_i}}(x_i),
 &\quad
\overline{\xi}_i\in\partial\Upsilon_i(x_i),
 \\[2mm]
\langle u_3^{(i)}(x_i,\cdot) \rangle_{\Upsilon_i(x_i)}
 & = &
0.
 &
\end{array}\right.
\end{equation}
Repeating the previous reasoning, we find
\begin{equation}\label{omega_probl_3}
 -  \pi \dfrac{d}{d{x_i}}\left(h_i^2(x_i)\frac{d\omega_1^{(i)}}{d{x_i}}(x_i)\right)
 =  \int\limits_{\Upsilon_i(x_i)}f_1^{(i)}(x_i,\overline{\xi}_i) \, d{\overline{\xi}_i},
\quad x_i\in I_\varepsilon^{(i)}, \quad i=1,2,3.
\end{equation}

Let us consider boundary-value problems for the functions  $u_k^{(i)}, \ k\geq 4, \ i=1,2,3 :$
\begin{equation}\label{regul_probl_k}
\left\{\begin{array}{rcll}
-\Delta_{\overline{\xi}_i}{u}_{k}^{(i)}(x_i,\overline{\xi}_i)
 & = &
\dfrac{d^{\,2}\omega_{k-2}^{(i)}}{d{x_i}^2}(x_i) + \dfrac{\partial^{\,2}u_{k-2}^{(i)}}{\partial{x_i}^2}(x_i,\overline{\xi}_i) + f_{k-2}^{(i)}(x_i,\overline{\xi}_i),
 &
\overline{\xi}_i\in\Upsilon_i(x_i),
 \\[2mm]
-\partial_{\nu_{\overline{\xi}_i}}{u}_{k}^{(i)}(x_i,\overline{\xi}_i)
 & = &
- \ h_i^\prime(x_i) \left(\dfrac{d\omega_{k-2}^{(i)}}{d{x_i}}(x_i) + \dfrac{\partial{u}_{k-2}^{(i)}}{\partial{x_i}}(x_i,\overline{\xi}_i)\right)
 &
 \\[2mm]
 &   &
+ \ \eta_{k-2}^{(i)}(x_i,\overline{\xi}_i) \, \varphi^{(i)}(x_i,\overline{\xi}_i),
 & \overline{\xi}_i\in\partial\Upsilon_i(x_i),
 \\[2mm]
\langle u_k^{(i)}(x_i,\cdot) \rangle_{\Upsilon_i(x_i)}
 & = &
0.
 &
\end{array}\right.
\end{equation}
Assume that all coefficients $u_2^{(i)},\dots,u_{k-1}^{(i)},\omega_0^{(i)},\dots,\omega_{k-3}^{(i)}$ of the expansion (\ref{regul}) are determined.
Then we can find  $u_k^{(i)}$ and $\omega_{k-2}^{(i)}$ from problem (\ref{regul_probl_k}). Indeed, it follows from the solvability condition of problem (\ref{regul_probl_k})  that $\omega_{k-2}^{(i)}$ must be a solution to the following ordinary differential equation:
$$
 -  \pi \dfrac{d}{d{x_i}}\left(h_i^2(x_i)\dfrac{d\omega_{k-2}^{(i)}}{d{x_i}}(x_i)\right)
 =  \int\limits_{\Upsilon_i(x_i)}f_{k-2}^{(i)}(x_i,\overline{\xi}_i) \, d{\overline{\xi}_i}
 -  \eta_{k-2}^{(i)}(x_i)
    \int\limits_{\partial\Upsilon_i(x_i)}
    \varphi^{(i)}(x_i,\overline{\xi}_i) \, dl_{\overline{\xi}_i}
$$
\begin{equation}\label{omega_probl_k}
 +  h_i^\prime(x_i)
    \int\limits_{\partial\Upsilon_i(x_i)}
    \dfrac{\partial{u}_{k-2}^{(i)}}{\partial{x_i}}(x_i,\overline{\xi}_i) \,
    dl_{\overline{\xi}_i},
    \quad x_i\in I_\varepsilon^{(i)}, \quad i=1,2,3, \quad k\geq4,
\end{equation}
where
\begin{equation}\label{eta_k}
\eta_k^{(i)} (x_i) =
\left\{\begin{array}{ll}
0, & \mbox{if $k$ is odd},
\\
\dfrac{(-1)^\frac{k}2 k! |h_i^\prime(x_i)|^k}{(1-k)((\frac{k}2)!)^2 4^\frac{k}2}, &
\mbox{if $k$ is even},
\end{array}\right.
\quad x_i\in I_\varepsilon^{(i)}, \quad k\in\Bbb N.
\end{equation}
\begin{remark}
Boundary conditions for the differential equations (\ref{omega_probl_2}), (\ref{omega_probl_3}) and (\ref{omega_probl_k}) are unknown in advance.
They will be determined in the process of construction of the asymptotics.
\end{remark}

Thus, the solution of problem (\ref{regul_probl_k}) is uniquely determined. Hence, the recursive procedure for the determination
of the coefficients of series (\ref{regul}) is uniquely solvable.

 \subsection{Boundary-layer part of the asymptotics}

In the previous subsection, we have considered the regular asymptotics taking into account the inhomogeneity of the right-hand side of the differential
equation in (\ref{probl}) and the boundary conditions on the lateral surfaces the thin cylinders $\Omega^{(i)}_\varepsilon, \ i=1, 2, 3.$ In what follows,
we construct the boundary-layer part of the asymptotics compensating the residuals of the regular one
at the base of $\Omega_\varepsilon^{(i)}.$

Substituting the series (\ref{prim+}) into (\ref{probl}) and collecting coefficients with the same powers of $\varepsilon$, we get the following mixed boundary-value problems:
\begin{equation}\label{prim+probl}
 \left\{\begin{array}{rcll}
  -\Delta_{\xi_i^*, \overline{\xi}_i} \Pi_k^{(i)}(\xi_i^*,\overline{\xi}_i) & =
   & 0,
   & \xi_i^*\in(0,+\infty), \quad \overline{\xi}_i\in\Upsilon_i(1),
   \\[2mm]
  -\partial_{\nu_{\overline{\xi}_i}} \Pi_k^{(i)}(\xi_i^*,\overline{\xi}_i) & =
   & 0,
   & \xi_i^*\in(0,+\infty), \quad \overline{\xi}_i\in\partial\Upsilon_i(1),
   \\[2mm]
  \Pi_k^{(i)}(0,\overline{\xi}_i) & =
   & \Phi_k^{(i)}(\overline{\xi}_i),
   & \overline{\xi}_i\in\Upsilon_i(1),
   \\[2mm]
  \Pi_k^{(i)}(\xi_i^*,\overline{\xi}_i) & \to
   & 0,
   & \xi_i^*\to+\infty, \quad \overline{\xi}_i\in\Upsilon_i(1),
 \end{array}\right.
\end{equation}\\[3mm]
where $\xi_i^* = \frac{1-x_i}{\varepsilon},$ $\overline{\xi}_i = \frac{\overline{x}_i}{\varepsilon},$
$$
\Phi_k^{(i)} = -\omega_{k}^{(i)}(1), \ k=0,1; \qquad
\Phi_k^{(i)}(\overline{\xi}_i)
 = -u_k^{(i)}(1,\overline{\xi}_i)
 -  \omega_{k}^{(i)}(1),
\quad k\geq 2, \quad k\in \Bbb N.
$$

Using the method of separation of variables, we determine the solution
\begin{equation}\label{view_solution}
\Pi_k^{(i)}(\xi_i^*,\overline{\xi}_i)
 =  a_{k,0}^{(i)}
 +  \sum\limits_{p=1}^{+\infty}a_{k,p}^{(i)} \,
    \Theta_p^{(i)}(\overline{\xi}_i) \,
    \exp({-\lambda_p^{(i)}\xi_i^*})
\end{equation}
of problem (\ref{prim+probl}) at a fixed index $k,$
where
$$
a_{k,p}^{(i)}
 =  \dfrac{1}{\|\Theta_p^{(i)}\|_{L^2(\Upsilon_i(1))}^2}
    \int\limits_{\Upsilon_i(1)}
    \Phi_k^{(i)}(\overline{\xi}_i)
    \Theta_p^{(i)}(\overline{\xi}_i) \, d\overline{\xi}_i,
$$
$$
a_{k,0}^{(i)}
 =  \dfrac{1}{|\Upsilon_i(1)|}
    \int\limits_{\Upsilon_i(1)}
    \Phi_k^{(i)}(\overline{\xi}_i) \, d\overline{\xi}_i
 = -\dfrac{1}{\pi h_i^2(1)}
    \int\limits_{\Upsilon_i(1)}
    u_k^{(i)} (1,\overline{\xi}_i) \, d\overline{\xi}_i
 -  \omega_{k}^{(i)}(1)
 = -\omega_{k}^{(i)}(1).
$$
Here $\Theta_0^{(i)}\equiv1, \ \Theta_p^{(i)}(\overline{\xi}_i)$ and $\{0=\lambda_0^{(i)}<\lambda_1^{(i)} \le \lambda_2^{(i)}\le\ldots\le\lambda_p^{(i)}\le\ldots\}$ are the eigenfunctions and eigenvalues of the spectral problem
\begin{equation}\label{spectral_problem}
\left\{\begin{array}{lcll}
- \Delta_{\overline{\xi}_i} \Theta^{(i)}  & = &  (\lambda^{(i)})^2\Theta^{(i)} &
\mbox{in} \ \Upsilon_i(1),
\\
\partial_{\nu_{\overline{\xi}_i}} \Theta^{(i)} & = & 0 &
\mbox{on} \ \partial\Upsilon_i(1).
\end{array}\right.
\end{equation}

It follows from the fourth condition in (\ref{prim+probl})  that coefficient $a_{k,0}^{(i)}$
must be equal to $0.$ As a result,
we arrive at  the following boundary conditions for the functions $\{\omega_{k}^{(i)}\}$:
\begin{equation}\label{bv_left}
\omega_{k}^{(i)}(1)=0, \quad k\in\Bbb{N}_0, \quad i=1,2,3.
\end{equation}

\begin{remark}
Since $\Phi_0^{(i)}\equiv\Phi_1^{(i)}\equiv0$, we conclude that $\Pi_{0}^{(i)}\equiv\Pi_{1}^{(i)}\equiv0, \quad i=1,2,3.$
Moreover, from representation (\ref{view_solution}) it follows the following asymptotic relations
\begin{equation}\label{as_estimates}
\Pi_k^{(i)}(\xi_i^*,\overline{\xi}_i)
 =  {\cal O}(\exp(- \lambda_1^{(i)}\xi_i^*))
\quad \mbox{as} \quad \xi_i^*\to+\infty, \quad i=1,2,3.
\end{equation}
\end{remark}


 \subsection{Inner part of the asymptotics}\label{inner_asymp}
 To obtain conditions for the functions $\{\omega_k^{(i)}\}, \ i=1,2,3$ at the point $0,$ we introduce the inner part~ (\ref{junc}) of the asymptotics in a neighborhood of the aneurysm $\Omega^{(0)}_\varepsilon$. For this we pass to the variables $\xi=\frac{x}{\varepsilon}.$ Then forwarding the parameter $\varepsilon$ to $0,$ we see  that the domain $\Omega_\varepsilon$ is transformed into the unbounded domain $\Xi$ that  is the union of the domain~$\Xi^{(0)}$ and three semibounded cylinders
$$
\Xi^{(i)}
 =  \{ \xi=(\xi_1,\xi_2,\xi_3)\in\Bbb R^3 \ :
    \quad  \ell<\xi_i<+\infty,
    \quad |\overline{\xi}_i|<h_i(0) \},
\qquad i=1,2,3,
$$
i.e., $\Xi$ is the interior of $\bigcup_{i=0}^3\overline{\Xi^{(i)}}$ (see Fig.~\ref{Fig-5}).

\begin{figure}[htbp]
\begin{center}
\includegraphics[width=9cm]{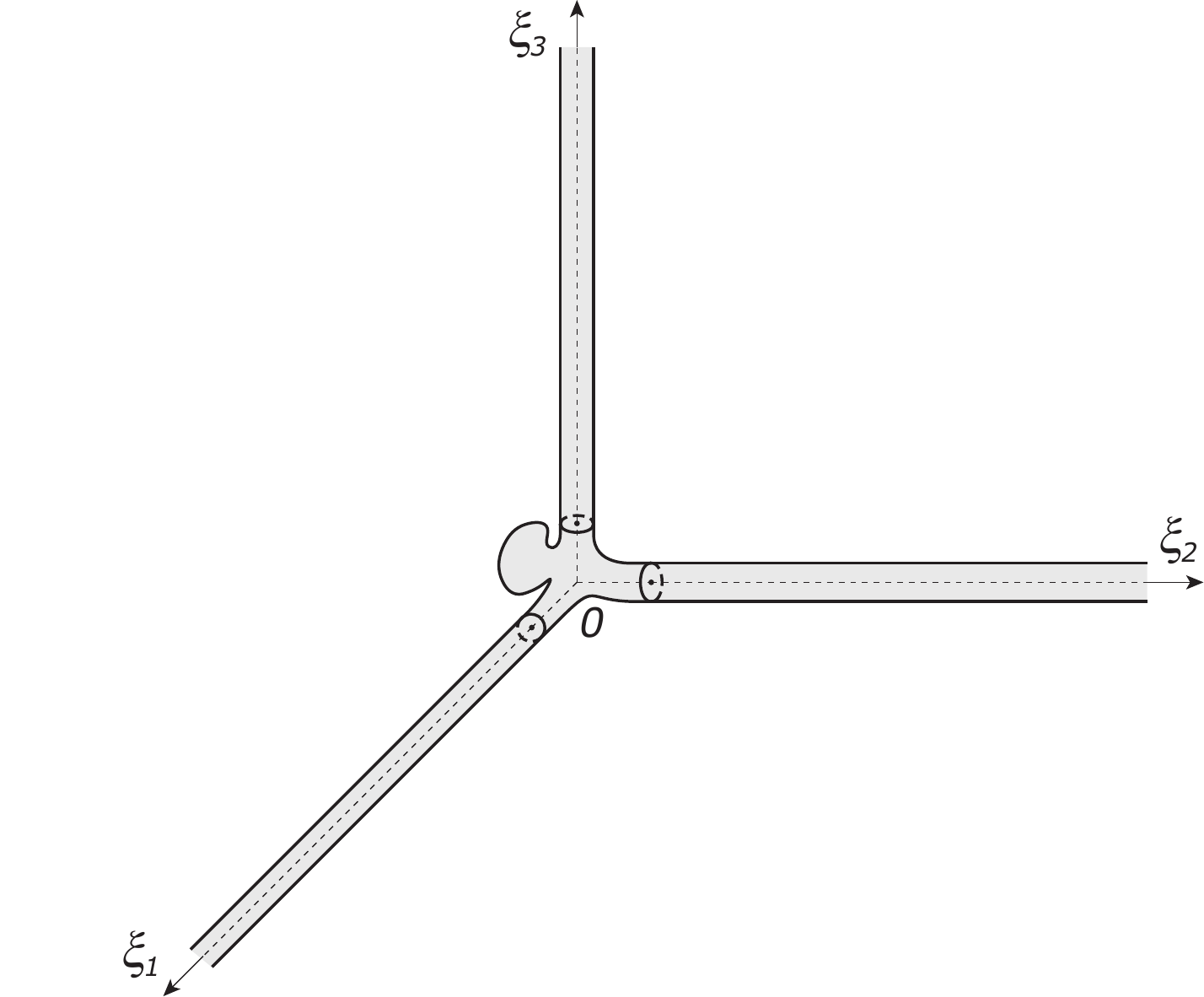}
\vskip - 10pt
\caption{Domain $\Xi$}\label{Fig-5}
\end{center}
\end{figure}

Let us introduce the following notation for parts of the boundary of  $\Xi$:
\begin{itemize}
  \item
 $\Gamma_i = \{ \xi\in\Bbb R^3 \ : \quad \ell<\xi_i<+\infty, \quad |\overline{\xi}_i|=h_i(0)\}, \quad i=1,2,3,$
  \item
 $\Gamma_0 = \partial\Xi \backslash \left(\bigcup_{i=1}^3 \Gamma_i \right).$
\end{itemize}

Substituting the series (\ref{junc}) into the problem (\ref{probl}) and equating coefficients at the same powers of $\varepsilon$, we derive the following relations for $\{N_k\}:$
\begin{equation}\label{junc_probl_n}
 \left\{\begin{array}{rcll}
  -\Delta_{\xi}{N_k}(\xi) & = &
   F_k(\xi),                  &
   \quad \xi\in\Xi,
\\[2mm]
   \partial_{{\boldsymbol \nu}_\xi}{N_k}(\xi) & = &
   0,                               &
   \quad \xi\in\Gamma_0,
\\[2mm]
  -\partial_{\,\overline{\nu}_{\overline{\xi}_i}}{N_k}(\xi) & = &
   B_{k}^{(i)}(\xi),                               &
   \quad \xi\in\Gamma_i, \quad i=1,2,3,
\\[2mm]
   N_k(\xi)                                                               & \sim &
   \omega^{(i)}_{k}(0) + \Psi^{(i)}_{k}(\xi),                                    &
   \quad \xi_i \to +\infty, \ \ \overline{\xi}_i \in \Upsilon_i(0), \quad i=1,2,3.
 \end{array}\right.
\end{equation}
Here
$$
F_0\equiv F_1\equiv 0, \quad
F_k(\xi) = \dfrac{(\xi, \nabla_x)^{k-2}f(0)}{(k-2)!}:=\frac{1}{(k-2)!} \Bigg(\sum_{i=1}^3 \xi_i \, \frac{\partial f(x)}{\partial x_i} \Bigg)^{k-2}\Big|_{x=0},  \quad
\xi\in\Xi,
$$
$$
B_{0}^{(i)} \equiv B_{1}^{(i)} \equiv 0, \qquad
B_{k}^{(i)}(\xi)
 =   \dfrac{\xi_i^{k-2}}{(k-2)!}
     \dfrac{\partial^{k-2}\varphi^{(i)}}{\partial{x}_i^{k-2}} (0,\overline{\xi}_i),
     \quad \xi\in\Gamma_i, \ \ i=1,2,3.
$$
The right hand sides in the differential equation and boundary conditions on $\{\Gamma_i\}$ of the problem~ (\ref{junc_probl_n}) are obtained with the help of the Taylor decomposition of the functions $f$ and $\varphi^{(i)}$ at the points $x=0$ and $x_i=0, \ i=1,2,3,$ respectively.

The fourth condition in (\ref{junc_probl_n}) appears by matching the regular and inner asymptotics in a neighborhood of the aneurysm, namely the asymptotics of
the terms $\{N_k\}$ as $\xi_i \to +\infty$ have to coincide with the corresponding asymptotics of  terms of the regular expansions (\ref{regul}) as $x_i =\varepsilon \xi_i \to +0, \ i=1,2,3,$ respectively.
Expanding each term of the regular asymptotics in the Taylor series at the points $x_i=0, \ i=1,2,3,$
and collecting the coefficients of the same powers of $\varepsilon,$  we get
\begin{equation}\label{Psi_k}
\begin{array}{c}
\Psi_{0}^{(i)} \equiv 0, \qquad
\Psi_{1}^{(i)}(\xi) = \xi_i\dfrac{d\omega_{0}^{(i)}}{dx}(0),  \quad  i=1,2,3,
\\[3mm]
\Psi_{k}^{(i)}(\xi)
 =   \sum\limits_{j=1}^{k} \dfrac{\xi_i^j}{j!}
     \dfrac{d^j\omega_{k-j}^{(i)}}{dx_i^j} (0)
 +   \sum\limits_{j=0}^{k-2}
     \dfrac{\xi_i^j}{j!}
     \dfrac{\partial^j{u_{k-j}^{(i)}}}{\partial x_i^j} (0, \overline{\xi}_i),
\quad  i=1,2,3, \ \ k \geq 2.
\end{array}
\end{equation}

A solution of the problem (\ref{junc_probl_n}) at $k\in \Bbb N$ is sought in the form
\begin{equation}\label{new-solution}
N_k(\xi) = \sum\limits_{i=1}^3 \Psi_{k}^{(i)}(\xi)\chi_i(\xi_i) + \widetilde{N}_k(\xi),
\end{equation}
where $ \chi_i \in C^{\infty}(\Bbb{R}_+),\ 0\leq \chi_i \leq1$ and
$$
\chi_i(\xi_i) =
\left\{\begin{array}{ll}
 0, & \text{if} \quad \xi_i \leq 1+\ell,
\\[2mm]
 1, & \text{if} \quad \xi_i \geq 2+\ell,
\end{array}\right. \qquad i=1,2,3.
$$
Then $\widetilde{N}_k$ has to be a  solution of the problem
\begin{equation}\label{junc_probl_general}
 \left\{\begin{array}{rcll}
  -\Delta_{\xi}{\widetilde{N}_k}(\xi) & = &
   \widetilde{F}_k(\xi),                  &
   \quad \xi\in\Xi,
\\[2mm]
   \partial_{\boldsymbol{\nu}_\xi}{\widetilde{N}_k}(\xi) & = &
   0,                                           &
   \quad \xi\in\Gamma_0,
\\[2mm]
  -\partial_{\boldsymbol{\nu}_{\overline{\xi}_i}}{\widetilde{N}_k}(\xi) & = &
   \widetilde{B}_{k}^{(i)}(\xi),                               &
   \quad \xi\in\Gamma_i, \quad i=1,2,3,
 \end{array}\right.
\end{equation}
and has to satisfy the following conditions:
\begin{equation}\label{junc_probl_general+cond}
   \widetilde{N}_k(\xi)  \rightarrow  \omega^{(i)}_{k}(0)
   \quad \text{as} \quad \xi_i \to +\infty, \ \ \overline{\xi}_i \in \Upsilon_i(0), \quad i=1,2,3,
\end{equation}
where
$$
\widetilde{F}_1(\xi)
 =  \sum\limits_{i=1}^3
    \Big(
    \xi_i\dfrac{d\omega_0^{(i)}}{dx_i}(0) \chi_i^{\prime\prime}(\xi_i)
 + 2\dfrac{d\omega_0^{(i)}}{dx_i}(0) \chi_i^{\prime}(\xi_i)
    \Big),
$$
$$
\widetilde{F}_k(\xi)
 =  \sum\limits_{i=1}^3
    \Bigg[
    \Big(
    \sum\limits_{j=1}^{k}
    \dfrac{\xi_i^{j}}{j!}
    \dfrac{d^j\omega_{k-j}^{(i)}}{dx_i^j}(0)
 +  \sum\limits_{j=0}^{k-2}
    \dfrac{\xi_i^{j}}{j!}
    \dfrac{\partial^j{u_{k-j}^{(i)}}}{\partial x_i^j}(0, \overline{\xi}_i)
    \Big)
    \chi_i^{\prime\prime}(\xi_i)
$$
$$
 + 2\Big(
    \sum\limits_{j=1}^{k}
    \dfrac{\xi_i^{j-1}}{(j-1)!}
    \dfrac{d^j\omega_{k-j}^{(i)}}{dx_i^j}(0)
 +  \sum\limits_{j=1}^{k-2}
    \dfrac{\xi_i^{j-1}}{(j-1)!}
    \dfrac{\partial^j{u_{k-j}^{(i)}}}{\partial x_i^j}(0, \overline{\xi}_i)
    \Big)
    \chi_i^{\prime}(\xi_i)
    \Bigg]
$$
$$
 +  \Big( 1 - \sum\limits_{i=1}^3 \chi_i(\xi_i) \Big)
    \dfrac{1}{(k-2)!}\left(\xi, \nabla_x\right)^{k-2} f(0)
$$
and
$$
 \widetilde{B}_{1}^{(i)} \equiv 0, \qquad
\widetilde{B}_{k}^{(i)} (\xi)
 =  \dfrac{\xi_i^{k-2}}{(k-2)!}
    \dfrac{\partial^{k-2}\varphi^{(i)}}{\partial{x}_i^{k-2}} (0,\overline{\xi}_i)
    \big( 1-\chi_i(\xi_i) \big),
\quad  i=1,2,3, \qquad k \geq 2.
$$

The existence of a solution to the  problem (\ref{junc_probl_general}) in the corresponding energetic space can be obtained from general results about the
asymptotic behavior of solutions to elliptic problems in domains with different exits to infinity \cite{Ko-Ol,La-Pa,Na-Pla,Naz96}.
We will use approach proposed in \cite{Naz96, ZAA99}.

 Let $C^{\infty}_{0,\xi}(\overline{\Xi})$ be a space of functions infinitely differentiable in $\overline{\Xi}$ and finite with
respect to  $\xi$, i.e.,
$$
\forall \,v\in C^{\infty}_{0,\xi}(\overline{\Xi}) \quad \exists \,R>0 \quad \forall \, \xi\in\overline{\Xi} \quad \xi_i \geq R, \ \ i=1,2,3 \, : \quad v(\xi)=0.
$$
We now define a  space  $\mathcal{H} := \overline{\left( C^{\infty}_{0,\xi}(\overline{\Xi}), \ \| \cdot \|_\mathcal{H} \right)}$, where
$$
\| v \|_\mathcal{H}
 =  \sqrt{\int_\Xi|\nabla v(\xi)|^2 \, d\xi + \int_\Xi |v(\xi)|^2 |\rho(\xi)|^2 \, d\xi \, } ,
$$
and the weight function  $ \rho \in C^{\infty}(\Bbb{R}^3),\ 0\leq \rho \leq1$ and
$$
\rho (\xi) =
\left\{\begin{array}{ll}
1,            & \mbox{if} \quad                   \xi \in \Xi^{(0)}, \\
|\xi_i|^{-1}, & \mbox{if} \quad \xi_i \geq \ell+1, \ \xi \in \Xi^{(i)}, \quad i=1,2,3.
\end{array}\right.
$$

\begin{definition}
A function $\widetilde{N}_k$ from the space $\mathcal{H}$ is called a weak solution of the problem (\ref{junc_probl_general}) if the identity
\begin{equation}\label{integr}
    \int\limits_{\Xi} \nabla \widetilde{N}_k \cdot \nabla v \, d\xi
 =  \int\limits_{\Xi} \widetilde{F}_k \, v \, d\xi
 -  \sum\limits_{i=1}^3 \int\limits_{\Gamma_i}
    \widetilde{B}^{(i)}_k \, v \, d\sigma_\xi.
\end{equation}
holds for all $v\in\mathcal{H}$.
\end{definition}

\begin{proposition}\label{tverd1}
   Let $\rho^{-1} \widetilde{F}_k\in L^2(\Xi),$
       $\rho^{-1} \widetilde{B}^{(i)}_{k} \in L^2(\Gamma_i), \quad i=1,2,3.$

Then there exist a weak solution of problem (\ref{junc_probl_general}) if and only if
\begin{equation}\label{solvability}
    \int\limits_{\Xi} \widetilde{F}_k \, d\xi
 =  \sum\limits_{i=1}^3 \, \int\limits_{\Gamma_i}
    \widetilde{B}^{(i)}_k \, d\sigma_\xi.
\end{equation}
This solution is defined up to an additive constant.
The additive constant  can be chosen to guarantee  the existence and uniqueness of a weak solution of problem (\ref{junc_probl_general}) with
the following differentiable asymptotics:
\begin{equation}\label{inner_asympt_general}
\widehat{N}_k(\xi)=\left\{
\begin{array}{rl}
    {\cal O}(\exp( - \gamma_1 \xi_1)) & \mbox{as} \ \ \xi_1\to+\infty,
\\[2mm]
    \delta_k^{(2)}
 +  {\cal O}(\exp( - \gamma_2 \xi_2)) & \mbox{as} \ \ \xi_2\to+\infty,
\\[2mm]
    \delta_k^{(3)}
 +  {\cal O}(\exp( - \gamma_3 \xi_3)) & \mbox{as} \ \ \xi_3\to+\infty,
\end{array}
\right.
\end{equation}
where $\gamma_i, \ i=1,2,3 $ are positive constants.
\end{proposition}

The constants $\delta_k^{(2)}$ and $\delta_k^{(3)}$ in (\ref{inner_asympt_general})
are defined as follows:
\begin{equation}\label{const_d_0}
\delta_k^{(i)}
 =      \int_{\Xi} \mathfrak{N}_i \, \widetilde{F}_k(\xi) \, d\xi
 + \sum\limits_{j=1}^3 \,
    \int_{\Gamma_j} \mathfrak{N}_i \, \widetilde{B}_k^{(j)}(\xi) \, d\sigma_\xi ,
\quad i=2,3, \ \  k\in\Bbb{N}_0,
\end{equation}
where $\mathfrak{N}_2$ and $\mathfrak{N}_3$ are special solutions to
the corresponding homogeneous problem
\begin{equation}\label{hom_probl}
  -\Delta_{\xi}\mathfrak{N} = 0 \ \ \text{in} \ \ \Xi, \qquad
  \partial_\nu \mathfrak{N} = 0 \ \ \text{on} \ \ \partial \Xi,
\end{equation}
for the problem (\ref{junc_probl_general}).

\begin{proposition}\label{tverd2}
The  problem (\ref{hom_probl}) has two linearly independent solutions $\mathfrak{N}_2$ and $\mathfrak{N}_3$ that do not belong to the space
$ {\cal H}$ and they have the following differentiable asymptotics:
\begin{equation}\label{inner_asympt_hom_solution_1}
\mathfrak{N}_2(\xi) = \left\{
\begin{array}{rl}
     - \dfrac{\xi_1}{\pi h_1^2(0)}
 +  {\cal O}(\exp( - \gamma_1 \xi_1)) & \mbox{as} \ \ \xi_1\to+\infty,
 \\[3mm]
    C_2^{(2)} + \dfrac{\xi_2}{\pi h_2^2(0)}
 +  {\cal O}(\exp( - \gamma_2 \xi_2)) & \mbox{as} \ \ \xi_2\to+\infty,
 \\[3mm]
    C_2^{(3)} + {\cal O}(\exp( - \gamma_3 \xi_3)) & \mbox{as} \ \ \xi_3\to+\infty,
\end{array}
\right.
\end{equation}
\begin{equation}\label{inner_asympt_hom_solution_2}
\mathfrak{N}_3(\xi) = \left\{
\begin{array}{rl}
     - \dfrac{\xi_1}{\pi h_1^2(0)}
 +     {\cal O}(\exp( - \gamma_1 \xi_1)) & \mbox{as} \ \ \xi_1\to+\infty,
 \\[3mm]
   C_3^{(2)} +  {\cal O}(\exp( - \gamma_2 \xi_2)) & \mbox{as} \ \ \xi_2\to+\infty,
 \\[3mm]
    C_3^{(3)} + \dfrac{\xi_3}{\pi h_3^2(0)}
 +   {\cal O}(\exp( - \gamma_3 \xi_3)) & \mbox{as} \ \ \xi_3\to+\infty,
\end{array}
\right.
\end{equation}

Any other solution to the homogeneous problem, which has polynomial growth at infinity, can be presented as a linear combination
$ \alpha_1 + \alpha_2 \mathfrak{N}_2 + \alpha_3 \mathfrak{N}_3.$
\end{proposition}
\begin{proof}
The solution $\mathfrak{N}_2$ is sought in the form of a sum
$$
\mathfrak{N}_2 (\xi) = - \dfrac{\xi_1}{\pi h_1^2(0)} \, \chi_1(\xi_1) + \dfrac{\xi_2}{\pi h_2^2(0)} \, \chi_2(\xi_2)  +  \widetilde{\mathfrak{N}}_2(\xi),
$$
where $\widetilde{\mathfrak{N}}_2 \in {\cal H}$ and $\widetilde{\mathfrak{N}}_2$ is the solution to the problem (\ref{junc_probl_general}) with right-hand sides
$$
\widetilde{F}_2^*(\xi) =
\left\{
\begin{array}{rl}
    \dfrac{1}{\pi h_1^2(0)}\left(\big(\xi_1 \, \chi_1^{\prime}(\xi_1)\big)' + \chi_1^{\prime}(\xi_1)\right),
  & \ \ \xi \in  \Xi^{(1)},
 \\[3mm]
  -\dfrac{1}{\pi h_2^2(0)}\left(\big(\xi_2 \, \chi_2^{\prime}(\xi_2)\big)' + \chi_2^{\prime}(\xi_2)\right),
  &  \ \ \xi \in  \Xi^{(2)},
 \\[2mm]
  0 \, ,
  & \ \ \xi \in \Xi^{(0)} \cup \Xi^{(3)},
\end{array}
\right.
$$
and $\widetilde{B}_2^*=0.$
It is easy to verify that the solvability condition (\ref{solvability}) is satisfied.
Thus, by virtue of Proposition 2.1 there exist a unique solution $\widetilde{\mathfrak{N}}_2 \in {\cal H}$  that has the asymptotics
$$
\widetilde{\mathfrak{N}}_2(\xi)
 =  (1- \delta_{1j}) C_2^{(j)}
 +  {\cal O}(\exp( - \gamma_j \xi_j))
\quad \mbox{as} \ \ \xi_j\to+\infty, \qquad j=1,2,3.
$$

Similar we can prove the existence of the solution $\mathfrak{N}_3$ with the asymptotics (\ref{inner_asympt_hom_solution_2}).

Obviously, that  $\mathfrak{N}_2$ and $\mathfrak{N}_3$ are linearly independent and
any other solution to the homogeneous problem, which has polynomial growth at infinity, can be presented as
$ \alpha_1 + \alpha_2 \mathfrak{N}_2 + \alpha_3 \mathfrak{N}_3.$
\end{proof}

\begin{remark}\label{remark_constant}
To obtain formulas (\ref{const_d_0}) for the constants $\delta_k^{(2)}$ and $\delta_k^{(3)},$  it is necessary to substitute
the functions $\widehat{N}_k, \mathfrak{N}_2$ and $\widehat{N}_k, \mathfrak{N}_3$ in  the second Green-Ostrogradsky formula
$$
\int_{\Xi_R} \big(\widehat{N} \, \Delta_\xi \mathfrak{N} - \mathfrak{N} \, \Delta_\xi \widehat{N}\big)\, d\xi = \int_{\partial \Xi_R}
\big(\widehat{N} \,\partial_{\nu_\xi} \mathfrak{N} - \mathfrak{N}\, \partial_{\nu_\xi}\widehat{N}\big)\, d\sigma_\xi
$$
respectively, and then pass to the limit as $R \to +\infty.$ Here $\Xi_R = \Xi \cap \{\xi : \ |\xi_i| < R, \ i=1, 2, 3\}.$
\end{remark}

\subsubsection{Limit problem and problems for $\{\omega_k\}$}
The problem (\ref{junc_probl_n})
at $k=0$ is as follows:
\begin{equation}\label{junc_probl_general-k=0}
 \left\{\begin{array}{rcll}
  -\Delta_{\xi}{{N}_0}(\xi) & = &
   0,                  &
   \quad \xi\in\Xi,
\\[2mm]
   \partial_{\boldsymbol{\nu}_\xi}{{N}_0}(\xi) & = &
   0,                                           &
   \quad \xi\in\Gamma_0,
\\[2mm]
  -\partial_{\boldsymbol{\nu}_{\overline{\xi}_i}}{{N}_0}(\xi) & = &
   0,                               &
   \quad \xi\in\Gamma_i, \quad i=1,2,3,
\\[2mm]
   {N}_0(\xi)                                                   & \sim &
   \omega^{(i)}_{0}(0),                                                          &
   \quad \xi_i \to +\infty, \ \ \overline{\xi}_i \in \Upsilon_i(0), \quad i=1,2,3,
 \end{array}\right.
\end{equation}
It is ease to verify that $\delta_0^{(2)}=\delta_0^{(3)}=0$ and $\widehat{N}_0 \equiv 0.$
Thus, this problem has a solution in $\mathcal{H}$ if and only if
\begin{equation}\label{trans0}
\omega_0^{(1)} (0) = \omega_0^{(2)} (0) = \omega_0^{(3)} (0);
\end{equation}
in this case  \ $N_0 \equiv  \widetilde{N}_0 \equiv \omega_0^{(1)} (0).$

In the problem (\ref{junc_probl_general}) at $k=1,$ we have $\widetilde{B}_{1}^{(i)} \equiv 0, \ i=1,2,3,$ and
$$
\widetilde{F}_1(\xi)
 =  \sum\limits_{i=1}^3
    \Big(
    \xi_i\dfrac{d\omega_0^{(i)}}{dx_i}(0) \chi_i^{\prime\prime}(\xi_i)
 + 2\dfrac{d\omega_0^{(i)}}{dx_i}(0) \chi_i^{\prime}(\xi_i)
    \Big).
$$
The solvability condition  (\ref{solvability}) reads in this case as follows:
\begin{equation}\label{transmisiont1}
 \pi h_1^2 (0) \frac{d\omega_{0}^{(1)}}{dx_1} (0) + \pi h_2^2 (0) \frac{d\omega_{0}^{(2)}}{dx_2} (0) + \pi h_3^2 (0) \frac{d\omega_{0}^{(3)}}{dx_3} (0) =  0.
\end{equation}

Thus for the functions $\{\omega_0^{(i)}\}_{i=1}^3$ that are the first  terms of the regular asymptotic expansion~(\ref{regul}), we obtain the following problem:
\begin{equation}\label{main}
 \left\{\begin{array}{rclr}
  - \pi \dfrac{d}{d{x_i}}\left(h_i^2(x_i)\dfrac{d\omega_0^{(i)}}{d{x_i}}(x_i)\right) & = &
    \widehat{F}_0^{(i)}(x_i), &                                x_i\in I_i, \quad i=1,2,3,
 \\[4mm]
    \omega_0^{(i)}(1) & = &
    0, &           i=1,2,3,
 \\[3mm]
    \omega_0^{(1)} (0) \ \, = \ \, \omega_0^{(2)} (0) & = &
    \omega_0^{(3)} (0), &
  \\[3mm]
    \sum\limits_{i=1}^3 \pi h_i^2 (0) \dfrac{d\omega_{0}^{(i)}}{dx_i} (0) & = &
    0, &
 \end{array}\right.
\end{equation}
where $I_i:=\{x: \ x_i\in (0,1), \ \overline{x_i}=(0,0)\}$ and
\begin{equation}\label{right-hand-side}
\widehat{F}_0^{(i)}(x_i)
 := \pi h_i^2(x_i) \, f(x) \, \big|_{\overline{x}_i=(0,0)}
 -  \int\limits_{\partial\Upsilon_i(x_i)}
    \varphi^{(i)} (x_i,\overline{\xi}_i) \, dl_{\overline{\xi}_i},
\quad \ x\in I_i.
\end{equation}
The problem (\ref{main}) is  called {\it limit problem} for problem (\ref{probl}).

Let us verify the solvability condition  (\ref{solvability}) for the problem (\ref{junc_probl_general}) at any fixed $k \in \Bbb N, \ k \ge 2$.
Taking into account  the  third relation in problems (\ref{regul_probl_2}), (\ref{regul_probl_3}) and (\ref{regul_probl_k}), the equality  (\ref{solvability}) can be re-written as follows:
$$
    \sum\limits_{i=1}^3
    \Bigg[ \,
    \pi h_i^2 (0) \, \int\limits_{\ell+1}^{\ell+2}
    \sum\limits_{j=1}^{k}
    \dfrac{\xi_i^{j-1}}{(j-1)!}
    \dfrac{d^j\omega_{k-j}^{(i)}}{dx_i^j}(0)
    \chi_i^{\prime}(\xi_i) \, d\xi_i
$$
$$
 +  \frac{1}{(k-2)!}\int\limits_{\ell}^{\ell+2} (1-\chi_i(\xi_i))
    \int\limits_{\Upsilon_i(0)} \left(\xi, \nabla_x\right)^{k-2}f(0) \, d\overline{\xi}_i  \, d\xi_i
 $$
$$
 -  \int\limits_{\ell}^{\ell+2} (1-\chi_i(\xi_i))
    \int\limits_{\partial\Upsilon_i (0)}
    \dfrac{\partial^{k-2}\varphi^{(i)}}{\partial{x}_i^{k-2}} (0,\overline{\xi}_i) \,
     dl_{\overline{\xi}_i}  \, d\xi_i
    \Bigg]
$$
$$
 +  \frac{1}{(k-2)!} \int\limits_{\Xi^{(0)}} \left( \xi, \nabla_x \right)^{k-2}f(0) \, d\xi \
 =  \ 0, \quad k\in\Bbb N, \ \ k \geq2.
$$
Whence, integrating by parts in the first integrals with regard to (\ref{omega_probl_2}), (\ref{omega_probl_3}) and (\ref{omega_probl_k}), we obtain the following relations for $\{\omega_k^{(i)}\}:$
\begin{equation}\label{transmisiont1}
    \sum\limits_{i=1}^3 \pi h_i^2 (0) \frac{d\omega_{k-1}^{(i)}}{dx_i} (0)
 =  d_{k-1}^*,
\end{equation}
where
\begin{multline}\label{const_d_*}
d_k^*
 =  \sum\limits_{i=1}^3
    \left(
    \sum\limits_{j=1}^k \frac{\ell^j}{j!}
    \int\limits_{\Upsilon_i(0)}
    \dfrac{\partial^{j-1}f_{k-j}^{(i)}}{\partial x_i^{j-1}}
    (0, \overline{\xi}_i) \ d\overline{\xi}_i
 -  \frac{\ell^k}{k!}
    \int\limits_{\partial\Upsilon_i(0)}
    \dfrac{\partial^{k-1}\varphi^{(i)}}{\partial{x}_i^{k-1}}
    (0, \overline{\xi}_i) \ dl_{\overline{\xi}_i}
    \right)
\\
 - \frac{1}{(k-1)!} \int\limits_{\Xi^{(0)}} \left( \xi, \nabla_x \right)^{k-1}f(0) \, d\xi,
\quad k\in\Bbb N,
\end{multline}
and $f_{k}^{(i)}$ is defined in (\ref{f_k}). Recall that $d_0^*=0.$

Hence, if the functions $\{\omega_{k-1}^{(i)}\}_{i=1}^3$ satisfy (\ref{transmisiont1}), then there exist a weak solution $\widetilde{N}_k$ of the problem~(\ref{junc_probl_general}). According to Proposition  \ref{tverd1}, it can be chosen in a unique way to guarantee the
asymptotics~(\ref{inner_asympt_general}).

However till now, we do not take into account the condition (\ref{junc_probl_general+cond}).
To satisfy this condition, we represent a weak solution of the problem (\ref{junc_probl_general}) in the following form:
$$
\widetilde{N}_k = \omega_k^{(1)}(0) + \widehat{N}_k.
$$
Taking into account the asymptotics (\ref{inner_asympt_general}), we have to put
\begin{equation}\label{trans1}
\omega_k^{(1)}(0) =  \omega_k^{(2)}(0) - \delta_k^{(2)} = \omega_k^{(3)}(0) - \delta_k^{(3)},  \quad k \in \Bbb N.
\end{equation}

As a result, we get the solution of the problem (\ref{junc_probl_n}) with the following asymptotics:
\begin{equation}\label{inner_asympt}
{N}_{k}(\xi)
 =  \omega_{k}^{(i)} (0) + \Psi_k^{(i)} (\xi)
 +  {\cal O}(\exp(-\gamma_i\xi_i))
\quad \mbox{as} \ \ \xi_i\to+\infty, \quad i=1,2,3.
\end{equation}

Let us denote by
$$
G_k(\xi)
 := \omega_{k}^{(i)}(0) + \Psi_k^{(i)}(\xi), \quad \xi\in\Xi^{(i)},
\quad i=1,2,3, \quad k\in \Bbb N.
$$
\begin{remark}\label{rem_exp-decrease}
Due to (\ref{inner_asympt}),
the functions $\{{N}_{k} - G_k\}_{k\in \Bbb N}$ are exponentially decrease as $\xi_i \to +\infty,$ $i=1,2,3.$
\end{remark}

Relations (\ref{trans1}) and (\ref{transmisiont1}) are the  first and second transmission conditions for the functions $\{\omega_k^{(i)}\}$ at $x=0.$
Thus, the functions  $\{\omega_k^{(1)}, \ \omega_k^{(2)}, \ \omega_k^{(3)} \}$ are determined from the problem
\begin{equation}\label{omega_probl*}
 \left\{\begin{array}{rclr}
  - \pi \dfrac{d}{d{x_i}}\left(h_i^2(x_i)\dfrac{d\omega_k^{(i)}}{d{x_i}}(x_i)\right) & = &
    \widehat{F}_k^{(i)}(x_i), &                                x_i\in I_i, \qquad i=1,2,3,
 \\[3mm]
    \omega_k^{(i)}(1) & = &
    0, &           i=1,2,3,
 \\[3mm]
    \omega_k^{(1)} (0) \ \, = \ \, \omega_k^{(2)} (0) - \delta_k^{(2)} & = &
                                   \omega_k^{(3)} (0) - \delta_k^{(3)}, &
 \\[3mm]
    \sum\limits_{i=1}^3 \pi \,h_i^2 (0) \,\dfrac{d\omega_{k}^{(i)}}{dx_i} (0) & = &
    d_k^*, &
 \end{array}\right.
\end{equation}
where
\begin{multline}\label{right-hand-side_k}
\widehat{F}_k^{(i)}(x_i)
 := \int\limits_{\Upsilon_i(x_i)} f_k^{(i)} (x_i,\overline{\xi}_i) \, d{\overline{\xi}_i}
 -  \eta_k^{(i)} (x_i)
    \int\limits_{\partial\Upsilon_i(x_i)}
    \varphi^{(i)} (x_i,\overline{\xi}_i) \, dl_{\overline{\xi}_i}
 \\
 +  h_i^\prime(x_i)
    \int\limits_{\partial\Upsilon_i(x_i)}
    \dfrac{\partial{u}_{k}^{(i)}}{\partial{x_i}} (x_i,\overline{\xi}_i) \,
    dl_{\overline{\xi}_i},
\quad
x\in I_i, \quad i=1,2,3, \quad k\in\Bbb N,
\end{multline}
and $f_k^{(i)}, \, \eta_k^{(i)}$ are defined in (\ref{f_k}) and (\ref{eta_k}) respectively; recall that ${u}_{1}^{(i)}\equiv 0, \ i=1, 2, 3.$

To solve the problem (\ref{omega_probl*}) at a fixed index $k,$ we make the following substitutions:
$$
\phi_k^{(1)}(x_1)=\omega_k^{(1)}(x_1),\qquad \phi_k^{(2)}(x_2) = \omega_k^{(2)}(x_2) - \delta_k^{(2)} (1- x_2),
\qquad
\phi_k^{(3)}(x_3)=\omega_k^{(3)}(x_3) - \delta_k^{(3)} (1- x_3).
$$
As a result for functions $\{\phi_k^{(i)}\}_{i=1}^3,$ we get the problem
\begin{equation}\label{tilde-omega_probl*}
 \left\{\begin{array}{rclr}
  - \pi \dfrac{d}{d{x_i}}\left(h_i^2(x_i)\dfrac{d\phi_k^{(i)}}{d{x_i}}(x_i)\right) & = &
    \widehat{\Phi}_k^{(i)}(x_i), &                                x_i\in I_i, \qquad i=1,2,3,
 \\[3mm]
    \phi_k^{(i)}(1) & = &
    0, &           i=1,2,3,
 \\[3mm]
    \phi_k^{(1)} (0) \ \, = \ \, \phi_k^{(2)} (0)  & = &
                                   \phi_k^{(3)} (0), &
 \\[3mm]
    \sum\limits_{i=1}^3 \pi \,h_i^2 (0) \,\dfrac{d\phi_{k}^{(i)}}{dx_i} (0) & = &
    d_k^* + \delta_k^{(2)} + \delta_k^{(3)}, &
 \end{array}\right.
\end{equation}
where $\widehat{\Phi}_k^{(1)}(x_1)= \widehat{F}_k^{(1)}(x_1),$ $\widehat{\Phi}_k^{(i)}(x_i)= \widehat{F}_k^{(i)}(x_i) -2\pi \delta_k^{(i)} h_i(x_i) \, h_i'(x_i), \ \ i=2, 3.$

Next for functions
$$
\widetilde{\phi}(x)=\left\{
                                                  \begin{array}{ll}
                                                    \phi^{(1)}(x_1), & \hbox{if} \ \ x_1 \in I_1,\\
                                                    \phi^{(2)}(x_2), & \hbox{if} \ \ x_2 \in I_2, \\
                                                    \phi^{(3)}(x_3), & \hbox{if} \ \ x_3 \in I_3,
                                                  \end{array}
                                                \right.
$$
defined on the graph $I_1\cup I_2\cup I_3,$
we introduce the Sobolev space
$$
{\cal H}_0 := \left\{ \widetilde{\phi}: \
\phi^{(i)} \in H^1(I_i), \ \ \phi^{(i)}(1) = 0, \ \ i=1, 2, 3, \ \ \hbox{and} \ \ \phi^{(1)}(0) = \phi^{(3)}(0) = \phi^{(3)}(0)
\right\}
$$
with the scalar product
$$
\langle \widetilde{\phi}, \widetilde{\psi}\rangle:= \sum_{i=1}^3 \pi \int_0^1 h_i^2(x_i)\, \frac{d\phi^{(i)}}{d{x_i}} \, \frac{d\psi^{(i)}}{d{x_i}}\, dx_i,
\qquad \widetilde{\phi},\, \widetilde{\psi} \in {\cal H}_0.
$$

A function $\widetilde{\phi}_k\in {\cal H}_0$ is called a weak solution to problem (\ref{tilde-omega_probl*}) if it satisfies the integral identity
$$
\langle \widetilde{\phi}_k, \widetilde{\psi}\rangle = \sum_{i=1}^3 \int_0^1 \Phi_k^{(i)}(x_i) \, \psi_i(x_i)\, dx_i + \big(d_k^* + \delta_k^{(2)} + \delta_k^{(3)}\big) \psi_1(0)
\qquad \forall \, \widetilde{\psi} \in {\cal H}_0.
$$

It follows from the Riesz representation theorem that problem (\ref{tilde-omega_probl*}) has a unique weak solution $\widetilde{\phi}_k$ at any fixed
index $k\in \Bbb N.$

\section{Complete asymptotic expansion and its justification}\label{justification}

{\it The first step.}
From the limit problem (\ref{main}) we uniquely determine the first terms $\{\omega_0^{(i)}\}_{i=1}^3$ of the regular asymptotic expansion~(\ref{regul}). Next we uniquely  determine the first term $N_0$ of the inner asymptotic expansion (\ref{junc}); it is a solution to the problem (\ref{junc_probl_general-k=0})
and $N_0=\omega_0^{(1)}(0).$
Then we rewrite problem (\ref{regul_probl_2}) for each index $i=1,2,3$ and a fixed $x_i\in I_i$ in the form
\begin{equation}\label{new_regul_probl_2}
\left\{\begin{array}{rclr}
-\Delta_{\overline{\xi}_i}{u}_{2}^{(i)} (x_i,\overline{\xi}_i)
 & = &
\dfrac{d^{\,2}\omega_0^{(i)}}{d{x_i}^2} (x_i) + f(x) \, \big|_{\overline{x}_i=(0,0)},
 & \overline{\xi}_i\in\Upsilon_i (x_i),
\\[2mm]
-\partial_{\nu_{\overline{\xi}_i}}{u}_{2}^{(i)}(x_i,\overline{\xi}_i)
 & = &
- \ h_i^\prime (x_i)\dfrac{d\omega_0^{(i)}}{d{x_i}} (x_i) + \varphi^{(i)} (x_i,\overline{\xi}_i),
 & \overline{\xi}_i\in\partial\Upsilon_i(x_i),
\\[2mm]
\langle u_2^{(i)} (x_i,\cdot) \rangle_{\Upsilon_i (x_i)}
 & = &
0.
\end{array}\right.
\end{equation}
It is easy to verify that the solvability condition for this problem is satisfied (see (\ref{omega_probl_2})). Therefore,
thanks to the third relation in (\ref{new_regul_probl_2}), there exists a unique solution to the problem (\ref{new_regul_probl_2}).

Now with the help of formulas (\ref{view_solution}), we determine the first terms $\Pi_2^{(i)}, \ i=1,2,3$ of the
boundary-layer expansions (\ref{prim+}), as solutions of problems  (\ref{prim+probl}) that can be rewritten as follows:
\begin{equation}\label{new_prim+probl_2}
 \left\{\begin{array}{rcll}
  -\Delta_{\xi_i^*, \overline{\xi}_i}
   \Pi_2^{(i)}(\xi_i^*,\overline{\xi}_i) & =
   & 0,
   & \xi_i^*\in(0,+\infty), \quad \overline{\xi}_i\in\Upsilon_i(1),
 \\[2mm]
  -\partial_{\nu_{\overline{\xi}_i}}
   \Pi_2^{(i)}(\xi_i^*,\overline{\xi}_i) & =
   & 0,
   & \xi_i^*\in(0,+\infty), \quad \overline{\xi}_i\in\partial\Upsilon_i(1),
 \\[2mm]
  \Pi_2^{(i)}(0,\overline{\xi}_i) & =
   & - u_2^{(i)} (1,\overline{\xi}_i),
   & \overline{\xi}_i\in\Upsilon_i(1),
 \\[2mm]
  \Pi_2^{(i)}(\xi_i^*,\overline{\xi}_i) & \to
   & 0,
   & \xi_i^*\to+\infty, \quad \overline{\xi}_i\in\Upsilon_i(1).
 \end{array}\right.
\end{equation}

{\it The second step.}
The second terms $\{\omega_1^{(i)}\}_{i=1}^3$ of the regular asymptotics (\ref{regul}) are founded  from the problem (\ref{omega_probl*}) that can be rewritten as follows:
\begin{equation}\label{omega_probl_1}
 \left\{\begin{array}{rclr}
  - \pi \dfrac{d}{d{x_i}}\left(h_i^2(x_i)\dfrac{d\omega_1^{(i)}}{d{x_i}}(x_i)\right) & = &
    \int\limits_{\Upsilon_i(x_i)}
    \sum\limits_{j=1}^3 (1-\delta_{ij}) \, \xi_j \, \frac{\partial}{\partial{x_j}}
    f(x) |_{\overline{x}_i=(0,0)} \, d{\overline{\xi}_i}, &
 \\[3mm]
    \omega_1^{(i)}(1) & = &
    0,   \qquad        i=1,2,3, &
\\[2mm]
    \omega_1^{(1)} (0) \ \, = \ \, \omega_1^{(2)} (0) - \delta_1^{(2)} & = &
                                   \omega_1^{(3)} (0) - \delta_1^{(3)}, &
\\[2mm]
    \sum\limits_{i=1}^3 \pi h_i^2 (0) \dfrac{d\omega_{1}^{(i)}}{dx_i} (0) & = &
    d_1^*. &
 \end{array}\right.
\end{equation}
The constants $\delta_1^{(2)}$ and $\delta_1^{(3)}$ are uniquely determined (see Remark~\ref{remark_constant}) by formula
\begin{equation}\label{delta_1}
\delta_1^{(i)}
 =  \int_{\Xi} \mathfrak{N}_i \,
    \sum\limits_{j=1}^3
    \Big(
    \xi_j\dfrac{d\omega_0^{(j)}}{dx_j}(0) \chi_j^{\prime\prime}(\xi_j)
 + 2\dfrac{d\omega_0^{(j)}}{dx_j}(0) \chi_j^{\prime}(\xi_j)
    \Big) \, d\xi,
\quad i=2,3.
\end{equation}
and the constant $d_1^*$ is determined from formula (\ref{const_d_*}) and
\begin{equation}\label{d_1^*}
d_1^*
 =  \ell \sum\limits_{i=1}^3
    \left(
    \pi h_i^2(0)
    f(0)
 -  \int\limits_{\partial\Upsilon_i(0)}
    \varphi^{(i)} (0, \overline{\xi}_i) \ dl_{\overline{\xi}_i}
    \right)
 -  |\Xi^{(0)}|\, f(0),
\end{equation}
where $|\Xi^{(0)}|$ is the volume of the aneurysm $\Xi^{(0)}$ (see Section~\ref{statement}).

Knowing $\{\omega_1^{(i)}\}_{i=1}^3,$ we can uniquely find the seconds terms of the regular asymptotics $\{{u}_{3}^{(i)}\}_{i=1}^3$ (series (\ref{regul})) and boundary asymptotics $\{\Pi_3^{(i)}\}_{i=1}^3$ (series (\ref{prim+})) from the problems
\begin{equation}\label{new_regul_probl_3}
\left\{\begin{array}{rclr}
-\Delta_{\overline{\xi}_i}{u}_{3}^{(i)}(x_i,\overline{\xi}_i)
 & = &
\dfrac{d^{\,2}\omega_1^{(i)}}{d{x_i}^2}(x_i) + \sum\limits_{j=1}^3 (1-\delta_{ij}) \, \xi_j \, \frac{\partial}{\partial{x_j}}
    f(x) |_{\overline{x}_i=(0,0)},
 &
\overline{\xi}_i\in\Upsilon_i(x_i),
 \\[2mm]
-\partial_{\nu_{\overline{\xi}_i}}{u}_{3}^{(i)}(x_i,\overline{\xi}_i)
 & = &
- \ h_i^\prime(x_i)\dfrac{d\omega_1^{(i)}}{d{x_i}}(x_i),
 &
\overline{\xi}_i\in\partial\Upsilon_i(x_i),
 \\[2mm]
\langle u_3^{(i)}(x_i,\cdot) \rangle_{\Upsilon_i(x_i)}
 & = &
0, &
\end{array}\right.
\end{equation}
and
\begin{equation}\label{new_prim+probl_3}
 \left\{\begin{array}{rcll}
  -\Delta_{\xi_i^*, \overline{\xi}_i}
   \Pi_3^{(i)}(\xi_i^*,\overline{\xi}_i) & =
   & 0,
   & \xi_i^*\in(0,+\infty), \quad \overline{\xi}_i\in\Upsilon_i(1),
   \\[2mm]
  -\partial_{\nu_{\overline{\xi}_i}}
   \Pi_3^{(i)}(\xi_i^*,\overline{\xi}_i) & =
   & 0,
   & \xi_i^*\in(0,+\infty), \quad \overline{\xi}_i\in\partial\Upsilon_i(1),
   \\[2mm]
  \Pi_3^{(i)}(0,\overline{\xi}_i) & =
   & - u_3^{(i)} (1,\overline{\xi}_i),
   & \overline{\xi}_i\in\Upsilon_i(1),
   \\[2mm]
  \Pi_3^{(i)}(\xi_i^*,\overline{\xi}_i) & \to
   & 0,
   & \xi_i^*\to+\infty, \quad \overline{\xi}_i\in\Upsilon_i(1),
 \end{array}\right.
\end{equation}
respectively.

The second term  ${N}_1$ of the inner asymptotic expansion (\ref{junc}) is the unique solution of the problem~(\ref{junc_probl_n}) that can now be rewritten in the form
\begin{equation}\label{new_junc_probl_1}
 \left\{\begin{array}{rcll}
  -\Delta_{\xi}{N_1}(\xi) & = &
   0,                         &
   \quad \xi\in\Xi,
 \\[2mm]
   \partial_{\nu_\xi}{N_1}(\xi) & = &
   0,                               &
   \quad \xi\in\Gamma_0,
 \\[2mm]
  -\partial_{\nu_{\overline{\xi}_i}}{N_1}(\xi) & = &
   0,                                              &
   \quad \xi\in\Gamma_i,              \quad i=1,2,3,
 \\[2mm]
   N_1(\xi)                                                               & \sim &
   \omega^{(i)}_{1}(0) + \xi_i\dfrac{d\omega_{0}^{(i)}}{dx}(0),                  &
   \quad \xi_i \to +\infty, \ \ \overline{\xi}_i \in \Upsilon_i(0), \quad i=1,2,3.
 \end{array}\right.
\end{equation}

Thus we have uniquely determined the first terms of the expansions (\ref{regul}), (\ref{prim+}) and (\ref{junc}).

{\it The inductive step.}
Assume that we have determined the coefficients $\omega_0^{(i)}, \omega_1^{(i)},\ldots,\omega_{n-3}^{(i)},$
$u_2^{(i)},  u_3^{(i)},\ldots,u_{n-1}^{(i)}$ of the series (\ref{regul}),
coefficients
$\Pi^{(i)}_2,  \Pi^{(i)}_3,\ldots,\Pi^{(i)}_{n-1}$ of the series (\ref{prim+}),
coefficients
${N}_1,\ldots,{N}_{n-3}$ of the series (\ref{junc}), constants $\delta_1^{(i)},\ldots,\delta_{n-3}^{(i)}$ and
$d_1^*,\ldots,d_{n-3}^*$.

Then we can find the solution $\{\omega_{n-2}^{(i)}\}_{i=1}^3$ of problem (\ref{omega_probl*})
with the constants $\delta_{n-2}^{(2)}, \, \delta_{n-2}^{(3)}$ (see (\ref{const_d_0})) in the first  transmission condition
and with the constant $d_{n-2}^*$ in the second transmision conditions.
It should be noted that constants $\{d_k^*\}_{k\in\Bbb N}$ depend only on $f$ and $\{\varphi^{(i)}\}_{i=1}^3$ and they are uniquely defined  by formulas (\ref{const_d_*}).

The coefficients $u_{n}^{(i)}, \ i=1,2,3,$ are determined as solutions of the following problems:
\begin{equation}\label{new_regul_probl_n}
\left\{\begin{array}{rclr}
-\Delta_{\overline{\xi}_i}{u}_{n}^{(i)}(x_i,\overline{\xi}_i)
 & = &
\dfrac{d^{\,2}\omega_{n-2}^{(i)}}{d{x_i}^2}(x_i) + \dfrac{\partial^{\,2}u_{n-2}^{(i)}}{\partial{x_i}^2}(x_i,\overline{\xi}_i) + f_{n-2}^{(i)}(x_i,\overline{\xi}_i),
 &
\overline{\xi}_i\in\Upsilon_i(x_i),
 \\[2mm]
-\partial_{\nu_{\overline{\xi}_i}}{u}_{n}^{(i)}(x_i,\overline{\xi}_i)
 & = &
- \ h_i^\prime(x_i) \left(\dfrac{d\omega_{n-2}^{(i)}}{d{x_i}}(x_i) + \dfrac{\partial{u}_{n-2}^{(i)}}{\partial{x_i}}(x_i,\overline{\xi}_i)\right)
 &
 \\[2mm]
 &   &
+ \ \eta_{n-2}^{(i)}(x_i,\overline{\xi}_i) \, \varphi^{(i)}(x_i,\overline{\xi}_i),
 &
\overline{\xi}_i\in\partial\Upsilon_i(x_i),
 \\[2mm]
\langle u_n^{(i)}(x_i,\cdot) \rangle_{\Upsilon_i(x_i)}
 & = &
0,
 &
\end{array}\right.
\end{equation}
where $f_k^{(i)}$ and $\eta_k^{(i)}$ are defined in (\ref{f_k}) and (\ref{eta_k}) respectively.
We note that solvability condition for problems (\ref{new_regul_probl_n}) takes place, because
$\langle u_{n-2}^{(i)}(x,\cdot) \rangle_{\Upsilon_i} = 0,$ $i=1,2,3.$ Here $x_i\in I_\varepsilon^{(i)}.$

Further we find the coefficients $\Pi_{n}^{(i)}, \ i=1,2,3$ of the boundary asymptotic expansions (\ref{prim+}) as solutions of problems (\ref{prim+probl}) that can be rewritten in the form
\begin{equation}\label{new_prim+probl_n}
 \left\{\begin{array}{rcll}
  -\Delta_{\xi_i^*, \overline{\xi}_i}
   \Pi_n^{(i)}(\xi_i^*,\overline{\xi}_i) & =
   & 0,
   & \xi_i^*\in(0,+\infty), \quad \overline{\xi}_i\in\Upsilon_i(1),
   \\[2mm]
  -\partial_{\nu_{\overline{\xi}_i}}
   \Pi_n^{(i)}(\xi_i^*,\overline{\xi}_i) & =
   & 0,
   & \xi_i^*\in(0,+\infty), \quad \overline{\xi}_i\in\partial\Upsilon_i(1),
   \\[2mm]
  \Pi_n^{(i)}(0,\overline{\xi}_i) & =
   & - u_n^{(i)} (1,\overline{\xi}_i),
   & \overline{\xi}_i\in\Upsilon_i(1),
   \\[2mm]
  \Pi_n^{(i)}(\xi_i^*,\overline{\xi}_i) & \to
   & 0,
   & \xi_i^*\to+\infty, \quad \overline{\xi}_i\in\Upsilon_i(1).
 \end{array}\right.
\end{equation}

Finally, we find the coefficient ${N}_{n-2}$ of the inner asymptotic expansion (\ref{junc}),
which is the unique solution of the problem (\ref{junc_probl_n}) that can now be rewritten in the form
\begin{equation}\label{junc_probl_n-2}
 \left\{
 \begin{array}{rcll}
  -\Delta_{\xi}{{N}_{n-2}}(\xi)            & = &
   \dfrac{(\xi, \nabla_x)^{n-4}f(0)}{(n-4)!} , &
   \quad \xi\in\Xi,
 \\[4mm]
   \partial_{\nu_\xi}{N_{n-2}}(\xi) & = &
   0,                                   &
   \quad \xi\in\Gamma_0,
 \\[2mm]
  -\partial_{\nu_{\overline{\xi}_i}}{N_{n-2}}(\xi)                            & = &
   \dfrac{\xi_i^{n-4}}{(n-4)!}
   \dfrac{\partial^{n-4}\varphi^{(i)}}{\partial{x}_i^{n-4}} (0,\overline{\xi}_i), &
   \quad \xi\in\Gamma_i,
 \\[2mm]
   {N}_{n-2}(\xi)                                                      & \sim &
   \sum\limits_{j=0}^{n-2} \dfrac{\xi_i^j}{j!}
   \dfrac{d^j\omega_{n-j-2}^{(i)}}{dx_i^j} (0)
  +\sum\limits_{j=0}^{n-4}
   \dfrac{\xi_i^j}{j!}
   \dfrac{\partial^j{u_{n-j-2}^{(i)}}}{\partial x_i^j} (0, \overline{\xi}_i), &
   \\[3mm]
   & \text{as} &\quad \xi_i\to+\infty, \quad \overline{\xi}_i\in\Upsilon_i(0) \quad  i=1,2,3. &
 \end{array}
 \right.
\end{equation}

Thus we  can successively determine all coefficients of series (\ref{regul}), (\ref{prim+}) and (\ref{junc}).

\subsection{Justification}

With the help of the series (\ref{regul}), (\ref{prim+}), (\ref{junc}) we construct the following series:
\begin{equation}\label{asymp_expansion}
    \sum\limits_{k=0}^{+\infty} \varepsilon^{k}
    \Big(
    \overline{u  }_k (x, \, \varepsilon, \, \alpha)
  + \overline{\Pi}_k (x, \, \varepsilon)
  + \overline{N  }_k (x, \, \varepsilon, \, \alpha)
    \Big),
\quad x\in\Omega_\varepsilon,
\end{equation}
where
$$
\overline{u}_k (x, \, \varepsilon, \, \alpha)
 := \sum\limits_{i=1}^3 \chi_\ell^{(i)} \left(\frac{x_i}{\varepsilon^\alpha}\right)
    \left( u_k^{(i)} \left( x_i, \frac{\overline{x}_i}{\varepsilon} \right)
 +  \omega_k^{(i)} (x_i) \right),
 \quad ( u_0\equiv u_1 \equiv 0 ),
$$
$$
\overline{\Pi}_k (x, \, \varepsilon)
 := \sum\limits_{i=1}^3 \chi_\delta^{(i)} (x_i) \,
    \Pi_k^{(i)} \left( \frac{1-x_i}{\varepsilon}, \frac{\overline{x}_i}{\varepsilon} \right),
 \quad ( \Pi_0\equiv \Pi_1 \equiv 0 ),
$$
$$
\overline{N}_k (x, \, \varepsilon, \, \alpha)
 := \left(1 - \sum\limits_{i=1}^3 \chi_\ell^{(i)} \left(\frac{x_i}{\varepsilon^\alpha}\right) \right)
    N_k \left( \frac{x}{\varepsilon} \right),
 \quad \left( N_0 \equiv \omega_0^{(1)}(0) \right),
$$
$\alpha$ is a fixed number from the interval $(\frac23, 1),$ $\chi_\ell^{(i)}, \ \chi_\delta^{(i)}$ are smooth cut-off functions defined by formulas
\begin{equation}\label{cut-off-functions}
\chi_\ell^{(i)} (x_i) =
\left\{\begin{array}{ll}
1, & \text{if} \ \ x_i \ge 3 \, \ell,
\\
0, & \text{if} \ \ x_i \le 2 \, \ell,
\end{array}\right.
\quad
\chi_\delta^{(i)}(x_i)=
\left\{\begin{array}{ll}
1, & \text{if} \ \ x_i \ge 1 -  \delta,
\\
0, & \text{if} \ \ x_i \le 1 - 2\delta,
\end{array}\right.
\quad i=1,2,3,
\end{equation}
and $\delta$ is a sufficiently small fixed positive number.

\begin{theorem}\label{mainTheorem}
 Series $(\ref{asymp_expansion})$ is the asymptotic expansion for the solution $u_\varepsilon$ to the boundary-value problem~$(\ref{probl})$
 in the Sobolev space $H^1(\Omega_\varepsilon),$ i.e., $\forall \,  m \in\Bbb{N} \ (m\geq2)  \ \ \exists \, {C}_m >0 \ \ \exists \, \varepsilon_0>0 \ \
      \forall\, \varepsilon\in(0, \varepsilon_0) :$
\begin{equation}\label{t0}
            \| \, u_\varepsilon - U_\varepsilon^{(m)}\|_{H^1(\Omega_\varepsilon)}
 \leq {C}_m \ \varepsilon^{\alpha (m\,-\frac{1}{2}) + \frac{1}{2} },
\end{equation}
 where
\begin{equation}\label{aaN}
U^{(m)}_{\varepsilon}(x)
 =  \sum\limits_{k=0}^{m} \varepsilon^{k}
    \Big(
    \overline{u  }_k (x, \, \varepsilon, \, \alpha)
  + \overline{\Pi}_k (x, \, \varepsilon)
  + \overline{N  }_k (x, \, \varepsilon, \, \alpha)
    \Big),
\quad x\in\Omega_\varepsilon,
\end{equation}
is the partial sum of $(\ref{asymp_expansion}).$
\end{theorem}

\begin{remark}
Hereinafter, all constants in inequalities are independent of the parameter~$\varepsilon.$
\end{remark}

\begin{proof}
Take an arbitrary $m\geq2, \ m\in\Bbb{N}$. Substituting the partial sum $ U^{(m)}_{\varepsilon}$
in the equations and the boundary conditions of problem~(\ref{probl}) and taking into account  relations (\ref{main})--(\ref{new_prim+probl_n}) for the coefficients of series (\ref{asymp_expansion}), we find
\begin{equation}\label{t1}
\Delta U^{(m)}_{\varepsilon}(x) + f(x)
 =  \sum\limits_{j=1}^{7} R^{(m)}_{\varepsilon, j}(x)
 =: R^{(m)}_\varepsilon(x),
\quad x\in\Omega_\varepsilon.
\end{equation}
where
\begin{equation}\label{t1_1}
R^{(m)}_{\varepsilon, 1}(x)
 =  \sum\limits_{k=m-1}^m \varepsilon^{k} \ \sum\limits_{i=1}^3
    \chi_\ell^{(i)} \left(\frac{x_i}{\varepsilon^\alpha}\right)
    \left(
    \dfrac{\partial^{2} u_{k}^{(i)}}{\partial{x}_i^2}
    \left(x_i,\frac{\overline{x}_i}{\varepsilon}\right)
 +  \dfrac{d^2\omega_{k}^{(i)}}{dx_i^2} (x_i)
    \right),
\end{equation}
\begin{multline}\label{t1_2}
R^{(m)}_{\varepsilon, 2}(x)
 =  \sum\limits_{k=1}^{m}\varepsilon^k \sum\limits_{i=1}^3
    \Bigg(
 - 2\varepsilon^{-1-\alpha}
    \frac{d\chi_\ell^{(i)}}{d\zeta_i}(\zeta_i)
    \bigg(
    \frac{\partial{N}_{k}}{\partial{\xi_i}}(\xi)
 -  \frac{\partial{G}_{k}}{\partial{\xi_i}}(\xi)
    \bigg)
\\
 -  \varepsilon^{-2\alpha}
    \frac{d^2\chi_\ell^{(i)}}{d\zeta_i^2}(\zeta_i)
    \Big( {N}_{k}(\xi) - G_{k}(\xi) \Big)
    \Bigg)
    \Bigg|_{\zeta_i=\frac{x_i}{\varepsilon^{\alpha}},\, \xi=\frac{x}{\varepsilon}},
\end{multline}
\begin{multline}\label{t1_3}
R^{(m)}_{\varepsilon, 3}(x)
 =  \sum\limits_{k=2}^{m} \varepsilon^{k} \sum\limits_{i=1}^3
    \bigg(
 - 2\varepsilon^{-1}\frac{d\chi_\delta^{(i)}}{dx_i} (x_i) \,
    \frac{\partial\Pi^{(i)}_{k}}{\partial{\xi}_i} ( \xi_i^*, \overline{\xi}_i )
\\
 +  \frac{d^2\chi_\delta^{(i)}}{dx_i^2} (x_i) \,
    \Pi^{(i)}_{k} ( \xi_i^*, \overline{\xi}_i )
    \bigg)
    \bigg|_{\xi_i^*=\frac{1-x_i}{\varepsilon},\, \overline{\xi}=\frac{\overline{x}_i}{\varepsilon}},
\end{multline}
\begin{multline}\label{t1_4}
R^{(m)}_{\varepsilon, 4}(x)
 =  \varepsilon^{m-1}
    \sum\limits_{i=1}^3 \chi_\ell^{(i)} \left(\frac{x_i}{\varepsilon^\alpha}\right)
    \frac{\varepsilon^{-1}}{(m-2)!}
\\
\times
    \Bigg[
     \sum\limits_{p=1}^3 (1-\delta_{ip}) \int\limits_0^{x_p}
     \bigg(
      \sum\limits_{s=1}^3 (1-\delta_{is})
      \frac{x_s-y_s}{\varepsilon} \frac{\partial}{\partial{y}_s}
     \bigg)^{m-2}
     \frac{\partial{f}}{\partial y_p}(y)
     \bigg|_{y_i=x_i, \ y_s=0, \, s \neq i,p} \, dy_p
\\
 +   \Bigg( \prod_{p=1, \, p \neq i}^3 \int\limits_0^{x_p} \Bigg)
     \bigg(
      \sum\limits_{s=1}^3 (1-\delta_{is})
      \frac{x_s-y_s}{\varepsilon} \frac{\partial}{\partial{y}_s}
     \bigg)^{m-2}
     \frac{\partial^2{f}}{\prod_{p=1, \, p \neq i}^3 \partial y_p}(y)
     \bigg|_{y_i=x_i} \, d\overline{y}_i
    \Bigg],
\end{multline}
\begin{multline}\label{t1_5}
R^{(m)}_{\varepsilon, 5}(x)
 =  \varepsilon^{\alpha(m-1)}
    \bigg( 1 - \sum\limits_{i=1}^3 \chi_\ell^{(i)} \left(\frac{x_i}{\varepsilon^\alpha}\right) \bigg)
    \frac{\varepsilon^{-\alpha}}{(m-2)!} \,
    \times \Bigg[
     \sum\limits_{p=1}^3 \int\limits_0^{x_p}
     \left( \frac{x-y}{\varepsilon^\alpha}, \, \nabla_y \right)^{m-2}
     \frac{\partial{f}}{\partial y_p}(y)
     \bigg|_{\overline{y}_p=(0,0)} \, dy_p
\\
 +   \sum\limits_{p=1}^3
     \Bigg( \prod_{s=1, \, s \neq p}^3 \int\limits_0^{x_s} \Bigg)
     \left( \frac{x-y}{\varepsilon^\alpha}, \, \nabla_y \right)^{m-2}
     \frac{\partial^2{f}}{\prod_{s=1, \, s \neq p}^3 \partial y_s}(y)
     \bigg|_{y_p=0} \, d\overline{y}_s
\\
 +   \int\limits_0^{x_1} \int\limits_0^{x_2} \int\limits_0^{x_3}
     \left( \frac{x-y}{\varepsilon^\alpha}, \, \nabla_y \right)^{m-2}
     \frac{\partial^3{f}}{\partial y_1 \partial y_2 \partial y_3}(y) \,
     dy_3 dy_2 dy_1
    \Bigg],
\end{multline}
\begin{multline}\label{t1_6}
R^{(m)}_{\varepsilon, 6}(x)
 =  \varepsilon^{\alpha(m-1)}
    \sum\limits_{i=1}^3 \ 2 \frac{d \chi_\ell^{(i)}}{d\zeta_i} (\zeta_i)
    \bigg|_{\zeta_i=\frac{x_i}{\varepsilon^\alpha}} \cdot
    \Bigg[
     \varepsilon^{(1-\alpha)m}
     \left(
      \frac{\partial u_{m}^{(i)}}{\partial x_i}
      \left( x_i, \frac{\overline{x}_i}{\varepsilon} \right)
 +    \frac{d \omega_{m}^{(i)}}{d x_i} (x_i)
     \right)
\\
 +   \sum\limits_{k=0}^{m-1}
     \frac{\varepsilon^{(1-\alpha)k} \, \varepsilon^{-\alpha}}{(m-k-1)!}
     \int\limits_0^{x_i}
     \left( \frac{x_i-y_i}{\varepsilon^\alpha} \right)^{m-k-1}
     \frac{\partial^{m-k+1}}{\partial y_i^{m-k+1}}
     \left(
      u_k^{(i)} \left( y_i, \frac{\overline{x}_i}{\varepsilon^\alpha} \right)
 +    \omega_k^{(i)} (y_i)
     \right)
     dy_i
    \Bigg],
\end{multline}
\begin{multline}\label{t1_7}
R^{(m)}_{\varepsilon, 7}(x)
 =  \varepsilon^{\alpha(m-1)}
    \sum\limits_{i=1}^3 \ \frac{d^2 \chi_\ell^{(i)}}{d\zeta_i^2} (\zeta_i)
    \bigg|_{\zeta_i=\frac{x_i}{\varepsilon^\alpha}}
\\
 \times\,
     \sum\limits_{k=0}^{m} \varepsilon^{(1-\alpha)k}
     \frac{\varepsilon^{-\alpha}}{(m-k)!}
     \int\limits_0^{x_i}
     \left( \frac{x_i-y_i}{\varepsilon^\alpha} \right)^{m-k}
     \frac{\partial^{m-k+1}}{\partial y_i^{m-k+1}}
     \left(
      u_k^{(i)} \left( y_i, \frac{\overline{x}_i}{\varepsilon^\alpha} \right)
 +    \omega_k^{(i)} (y_i)
     \right)
     dy_i.
\end{multline}

From (\ref{t1_1}) we conclude that
\begin{equation}\label{t2_1}
\exists\, \check{C}_m>0  \ \ \exists \, \varepsilon_0>0 \ \
\forall\, \varepsilon\in (0, \varepsilon_0) : \quad
      \sup\limits_{x\in\Omega_{\varepsilon}}
      \left|R_{\varepsilon, 1}^{(m)}(x)\right|
 \leq \check{C}_m \varepsilon^{m-1}.
\end{equation}
Due to the exponential decreasing of functions $\{{N}_{k} - {G}_{k},  \Pi^{(i)}_{k} \}$ (see Remark~\ref{rem_exp-decrease} and (\ref{as_estimates})) and the fact that the support of the derivatives of cut-off function $\chi_\ell^{(i)}$ belongs to the set $\{x_i: 2 \ell \varepsilon^\alpha \le x_i \le 3 \ell \varepsilon^\alpha\},$ we arrive that
\begin{equation}\label{t2_2}
      \sup\limits_{x\in\Omega_{\varepsilon}}
      \left|R_{\varepsilon, 2}^{(m)}(x)\right|
 \leq \check{C}_m \varepsilon^{-1-\alpha}
      \exp{
       \left(
       -\frac{2 \ell}{\varepsilon^{1-\alpha}} \
        \min\limits_{i=1,2,3}{\gamma_i}
       \right)
          },
\end{equation}
similarly we obtain that
\begin{equation}\label{t2_3}
      \sup\limits_{x\in\Omega_{\varepsilon}}
      \left|R_{\varepsilon, 3}^{(m)}(x)\right|
 \leq \check{C}_m \varepsilon^{-1}
      \exp{
       \left(
       -\frac{\delta}{\varepsilon}\
        \min\limits_{i=1,2,3}{\lambda_1^{(i)}}
       \right)
          }.
\end{equation}
We calculate terms $R_{\varepsilon, j}^{(m)}, \ \ j=4,5,6,7$ with the help of the Taylor formula with the integral remaining term for functions $f$, $\{\omega_k\}$ and $\{u_{k}\}$ at the point $x_i=0$. It is easy to check that
\begin{equation}\label{t2_4567}
      \sup\limits_{x\in\Omega_{\varepsilon}}
      \left| R_{\varepsilon, 4}^{(m)}(x)\right|
 \leq \check{C}_m\varepsilon^{m-1},
\qquad
      \sup\limits_{x\in\Omega_{\varepsilon}}
      \left| R_{\varepsilon, j}^{(m)}(x)\right|
 \leq \check{C}_m\varepsilon^{\alpha(m-1)},
\quad
j=5,6,7.
\end{equation}

The partial sum leaves the following residuals on the boundary of $\Omega_\varepsilon:$
$$
\begin{array}{rclll}
    \partial_\nu{U_\varepsilon^{(m)}}(x)
 +  {\varepsilon\varphi^{(i)}}
    \Big( x_i, \dfrac{\overline{x}_i}{\varepsilon} \Big)
& = &
     \breve{R}_{\varepsilon, {(i)}}^{(m)}(x),
& x\in\Gamma_\varepsilon^{(i)},                        & i=1,2,3,
\\[2mm]
    U_\varepsilon^{(m)}(x)
& = &
    0, & x\in\Upsilon_\varepsilon^{(i)}(1), & i=1,2,3,
\\[2mm]
    \partial_\nu{U_\varepsilon^{(m)}}(x)
& = &
    0, & x\in\Gamma_\varepsilon^{(0)}, &
\end{array}
$$
where $\breve{R}_{\varepsilon, {(i)}}^{(m)}(x):=\sum\limits_{j=8}^9\breve{R}_{\varepsilon, j, {(i)}}^{(m)},$
\begin{multline}\label{breve_R_8}
\breve{R}_{\varepsilon, 8, (i)}^{(m)}(x)
 =  \frac{\varepsilon}{\sqrt{ 1 + \varepsilon^2 |h_i^\prime(x_i)|^2 \, }}
    \chi_\ell^{(i)} \left(\frac{x_i}{\varepsilon^\alpha}\right)
        \cdot
    \Bigg[
 -   \sum\limits_{k=m-1}^m \varepsilon^k h_i^\prime (x_i)
     \left(
      \frac{\partial u_k^{(i)}}{\partial x_i}
      \Big( x_i, \frac{\overline{x}_i}{\varepsilon} \Big)
 +    \frac{d\omega_k^{(i)}}{dx_i} (x_i)
     \right)
\\
 +    \varepsilon^{ 2 \lceil \frac{m}{2} \rceil }
     \frac{(-1)^{\lceil \frac{m}{2} \rceil} (2 \lceil \frac{m}{2} \rceil)! \lceil \frac{m}{2} \rceil \varepsilon^{-2}}
          {(1 - 2 \lceil \frac{m}{2} \rceil)(\lceil \frac{m}{2} \rceil !)^2 4^{\lceil \frac{m}{2} \rceil}} \
    \varphi^{(i)} \Big( x_i, \dfrac{\overline{x}_i}{\varepsilon} \Big)
          \int\limits_{0}^{\varepsilon^2 |h_i^{\prime}(x_i)|^2}
     \left(
      \frac{\varepsilon^2 |h_i^{\prime}(x_i)|^2 - t}{\varepsilon^2}
     \right)^{\lceil \frac{m}{2} \rceil - 1}
     (1+t)^{\frac12 - \lceil \frac{m}{2} \rceil} dt
       \Bigg],
\end{multline}
\begin{equation}\label{breve_R_9}
\breve{R}_{\varepsilon, 9, (i)}^{(m)}(x)
 =  \varepsilon^{1+\alpha(m-1)}
    \bigg( 1 - \chi_\ell^{(i)} \left(\frac{x_i}{\varepsilon^\alpha}\right) \bigg)
    \frac{\varepsilon^{-\alpha}}{(m-2)!}
    \,   \int\limits_0^{x_i} \left( \frac{x_i-y_i}{\varepsilon^\alpha} \right)^{m-2}
    \frac{\partial^{m-1}\varphi^{(i)}}{\partial y_i^{m-1}}
    \Big( y_i, \frac{\overline{x}_i}{\varepsilon} \Big) \ dy_i.
\end{equation}
In (\ref{breve_R_8}) the symbol $\lceil \eta \rceil $ denotes the ceiling of the number $\eta.$
It follows from (\ref{breve_R_8}) and (\ref{breve_R_9}) that there exist positive constants $\overline{C}_m$ and   $\overline{\varepsilon}_0$
such that for all $i=1,2,3$ and
\begin{equation}\label{t2_89}
\forall \, \varepsilon\in(0, \overline{\varepsilon}_0): \quad
      \sup\limits_{x\in\Gamma_\varepsilon^{(i)}}
      \left|\breve{R}_{\varepsilon, 8, (i)}^{(m)}(x)\right|
 \leq \overline{C}_m\varepsilon^{m}, \quad
      \sup\limits_{x\in\Gamma_\varepsilon^{(i)}}
      \left|\breve{R}_{\varepsilon, 9, (i)}^{(m)}(x)\right|
 \leq \overline{C}_m\varepsilon^{1+\alpha(m-1)}.
\end{equation}

Using (\ref{t2_1})~--~(\ref{t2_4567}) and (\ref{t2_89}), we obtain the following estimates:
\begin{equation}\label{t3_14}
      \left\|R_{\varepsilon, j}^{(m)}\right\|_{L^2 (\Omega_\varepsilon)}
 \leq \check{C}_m \sqrt{\pi \sum_{i=1}^3 \max\limits_{x_i\in I_i}h_i^2(x_i)\ }
      \ \varepsilon^{m},
\quad j=1,4,
\end{equation}
\begin{equation}\label{t3_2}
      \left\|R_{\varepsilon, 2}^{(m)}\right\|_{L^2 (\Omega_\varepsilon)}
 \leq \check{C}_m \sqrt{\pi \ell \sum_{i=1}^3 h_i^2(0)\ }
      \ \varepsilon^{-\frac{\alpha}{2}} \
      \exp{
       \left(
       -\frac{2 \ell}{\varepsilon^{1-\alpha}}\
        \min\limits_{i=1,2,3}{\gamma_i}
       \right)
          },
\end{equation}
\begin{equation}\label{t3_3}
      \left\|R_{\varepsilon, 3}^{(m)}\right\|_{L^2 (\Omega_\varepsilon)}
 \leq \check{C}_m \sqrt{\pi \sum_{i=1}^3 h_i^2(1)\ }
      \ \delta^{\frac{1}{2}} \
      \exp{
       \left(
       -\frac{\delta}{\varepsilon} \
        \min\limits_{i=1,2,3}{\lambda_1^{(i)}}
       \right)
          },
\end{equation}
\begin{equation}\label{t3_5}
      \left\|R_{\varepsilon, 5}^{(m)}\right\|_{L^2 (\Omega_\varepsilon)}
 \leq \check{C}_m \sqrt{|\Xi^{(0)}| + 3 \pi \ell \sum_{i=1}^3 h_i^2(0)\ }
      \ \varepsilon^{\alpha(m-\frac12)+1},
\end{equation}
\begin{equation}\label{t3_67}
      \left\|R_{\varepsilon, j}^{(m)}\right\|_{L^2 (\Omega_\varepsilon)}
 \leq \check{C}_m \sqrt{\pi \ell \sum_{i=1}^3 h_i^2(0)\ }
      \ \varepsilon^{\alpha(m-\frac12)+1}, \quad j=6,7,
\end{equation}
\begin{equation}\label{t3_8}
      \left\|\breve{R}_{\varepsilon, 8, (i)}^{(m)}\right\|_{L^2 (\Gamma_\varepsilon^{(i)})}
 \leq \overline{C}_m \sqrt{2 \pi \max\limits_{x_i\in I_i} h_i(x_i) \ }
      \ \varepsilon^{m+\frac12}, \quad i=1,2,3,
\end{equation}
\begin{equation}\label{t3_9}
      \left\|\breve{R}_{\varepsilon, 9, (i)}^{(m)}\right\|_{L^2 (\Gamma_\varepsilon^{(i)})}
 \leq \overline{C}_m \sqrt{6 \pi \ell h_i(0) \ }
      \ \varepsilon^{\alpha(m-\frac12)+\frac32}, \quad i=1,2,3.
\end{equation}
Thus, the difference  $W_\varepsilon := u_\varepsilon - U_\varepsilon^{(m)}$ satisfies the following relations:
\begin{equation}\label{nevyazka}
\left\{\begin{array}{rclll}
 - \Delta W_\varepsilon        & = & R_\varepsilon^{(m)}               &
   \mbox{in} \ \Omega_\varepsilon, &
\\[2mm]
 - \partial_\nu W_\varepsilon  & = & \breve{R}_{\varepsilon,(i)}^{(m)} &
   \mbox{on} \ \Gamma_\varepsilon^{(i)}, &                      i=1,2,3,
\\[2mm]
   W_\varepsilon               & = & 0                                 &
   \mbox{on} \ \Upsilon_\varepsilon^{(i)}(1), &                 i=1,2,3,
\\[2mm]
   \partial_\nu{W_\varepsilon} & = & 0                                 &
   \mbox{on} \ \Gamma_\varepsilon^{(0)}, &
\end{array}\right.
\end{equation}

From (\ref{nevyazka}) we derive the following integral relation:
$$
    \int \limits_{\Omega_\varepsilon} {|\nabla W_\varepsilon |}^2 dx
 =  \int \limits_{\Omega_\varepsilon} R_\varepsilon^{(m)} \, W_\varepsilon \,dx
 -  \sum \limits_{i=1}^3 \
    \int \limits_{\Gamma_\varepsilon^{(i)}}
    \breve{R}_{\varepsilon, (i)}^{(m)} \, W_\varepsilon \, d\sigma_x.
$$
In view of the Friedrichs inequality  and estimates  (\ref{t3_14})~--~(\ref{t3_9}), this yields the following inequality:
$$
      \int \limits_{\Omega_\varepsilon} {|\nabla W_\varepsilon |}^2 dx
 \leq \check{c}_m \
      \varepsilon^{\alpha(m-\frac12)+\frac12}
      \| W_\varepsilon \|_{L^2(\Omega_\varepsilon)}
 +    \overline{c}_m \ \varepsilon^{\alpha(m-\frac12)+1}
      \sum\limits_{i=1}^3
      \| W_\varepsilon \|_{L^2(\Gamma_\varepsilon^{(i)})}
$$
$$
 \leq {C}_m \, \varepsilon^{\alpha(m-\frac12)+\frac12}
      \|\nabla W_\varepsilon\|_{L^2(\Omega_\varepsilon)}.
$$
This, in turn, means the asymptotic estimate (\ref{t0}) and proves the theorem.
\end{proof}

\begin{corollary}\label{corollary1}
The differences between the solution $u_\varepsilon$ of problem $(\ref{probl})$ and the partial sums $U_\varepsilon^{(0)}, \ U_\varepsilon^{(1)}$ $($see $(\ref{aaN}))$ admit the following asymptotic estimates:
\begin{equation}\label{t5}
 \| \, u_\varepsilon - U_\varepsilon^{(0)} \|_{H^1(\Omega_\varepsilon)} \leq \widetilde{C}_0 \,
 \varepsilon^{1+\frac\alpha2}, \qquad
\| \, u_\varepsilon - U_\varepsilon^{(0)} \|_{L^2(\Omega_\varepsilon)} \leq \widetilde{C}_0 \,
 \varepsilon^{\frac32\alpha+\frac12},
\end{equation}
\begin{equation}\label{t6}
 \| \, u_\varepsilon - U_\varepsilon^{(1)} \|_{H^1(\Omega_\varepsilon)} \leq \widetilde{C}_0 \,
 \varepsilon^{1+\alpha},
\end{equation}
where $\alpha$ is a fixed number from the interval $(\frac23, 1).$

In thin cylinders
$\Omega_{\varepsilon,\alpha}^{(i)} :=
 \Omega_\varepsilon^{(i)} \cap \big\{ x\in \Bbb{R}^3 : \
 x_i\in I_{\varepsilon, \alpha}^{(i)}:= (3\ell\varepsilon^\alpha, 1) \big\}, \ i=1,2,3,$
the following estimates hold:
\begin{equation}\label{t7}
 \| \, u_\varepsilon - \omega_0^{(i)} \|_{H^1(\Omega_{\varepsilon,\alpha}^{(i)})} \leq \widetilde{C}_1 \,
 \varepsilon^{2}, \ \ i=1,2,3,
\end{equation}
where
$\{\omega_0^{(i)}\}_{i=1}^3$ is the solution of the limit problem~$(\ref{main}).$

In the  neighbourhood
$\Omega^{(0)}_{\varepsilon, \ell} :=
 \Omega_\varepsilon\cap \big\{ x : \ \ x_i<2\ell\varepsilon, \ i=1,2,3 \big\}$
of the aneurysm $\Omega^{(0)}_{\varepsilon},$ we get estimates
 \begin{equation}\label{t-joint0}
     \| \, \nabla_{x}u_\varepsilon  - \nabla_{\xi} \, N_1\|_{L^2(\Omega^{(0)}_{\varepsilon, \ell})}
 \le \| \,
      u_\varepsilon - \omega_0^{(i)}(0) - \varepsilon \, N_1
     \|_{H^1(\Omega^{(0)}_{\varepsilon, \ell})}
 \le \widetilde{C}_4 \, \varepsilon^{\frac52},
\end{equation}
 \end{corollary}

\begin{proof} Denote by $\chi_{\ell,\alpha,\varepsilon}^{(i)}(\cdot) := \chi_\ell^{(i)}(\tfrac{\cdot}{\varepsilon^\alpha})$ (the function $\chi_\ell^{(i)}$
is determined in (\ref{cut-off-functions})).
Using the smoothness of the functions $\{\omega_k^{(i)}\}$ and the exponential decay of  the functions
$\{ N_k-G_k \}$ and $\{\Pi_{k}^{(i)}\}, \ i=1,2,3,$ at infinity, we deduce the inequality
(\ref{t5}) from  estimate (\ref{t0}) at $m=2$:
$$
     \left\| \, u_\varepsilon-U^{(0)}_{\varepsilon}\right\|_{H^1(\Omega_\varepsilon)}
 \le \left\| \, u_\varepsilon-U^{(2)}_{\varepsilon}\right\|_{H^1(\Omega_\varepsilon)}
   + \Bigg\| \,
      \varepsilon \sum_{i=1}^3
      \bigg(
       \chi_{\ell,\alpha,\varepsilon}^{(i)} \, \omega_1^{(i)}
   +   \bigg( 1 - \sum_{i=1}^3 \chi^{(i)}_{\ell,\alpha,\varepsilon} \bigg) \, N_1
      \bigg)
$$
$$
   +  \varepsilon^{2}
      \left(
       \sum_{i=1}^3
       \Big(
        \chi^{(i)}_{\ell,\alpha,\varepsilon} \,(u_{2}^{(i)} + \omega_{2}^{(i)})
   +    \chi^{(i)}_\delta \Pi^{(i)}_{2}
       \Big)
   +   \bigg( 1 - \sum_{i=1}^3 \chi_{\ell,\alpha,\varepsilon}^{(i)} \, \bigg) {N}_{2}
      \right)
     \Bigg\|_{H^1(\Omega_\varepsilon)}  \le C_2  \varepsilon^{\frac32\alpha+\frac12}
$$
$$
  + \sum\limits_{i=1}^3
     \bigg\| \,
      \varepsilon \left( \chi_{\ell,\alpha,\varepsilon}^{(i)} \omega_1
   +  \big(1-\chi_{\ell,\alpha,\varepsilon}^{(i)}\big) N_1 \right)
   +  \varepsilon^{2} \left( \chi_{\ell,\alpha,\varepsilon}^{(i)}(u_{2} + \omega_{2})
   +  \big(1-\chi_{\ell,\alpha,\varepsilon}^{(i)}\big) N_2 \right)
     \bigg\|_{H^1(\Omega^{(i)}_\varepsilon)}
$$
$$
   + \varepsilon   \left\| N_1 \right\|_{H^1(\Omega_\varepsilon^{(0)})}
   + \varepsilon^2 \left\| N_2 \right\|_{H^1(\Omega_\varepsilon^{(0)})}
   + \varepsilon^2 \sum_{i=1}^3
     \left\| \, \chi_\delta^{(i)} \Pi_2^{(i)} \right\|_{H^1(\Omega_\varepsilon^{(i)})}
$$
$$
 \le C_2 \, \varepsilon^{\frac32\alpha+\frac12}
   + \sum_{i=1}^3
     \bigg\|
      \big(1-\chi_{\ell,\alpha,\varepsilon}^{(i)}\big) \,
      \Big(
       x_i \frac{d\omega_0^{(i)}}{dx_i}(0)
   +   \frac{x_i^2}{2} \frac{d^2\omega_0^{(i)}}{dx_i^2}(0)
       \Big)
     \bigg\|_{H^1(\Omega_\varepsilon^{(i)})}
$$
$$
   + \varepsilon \sum_{i=1}^3
     \bigg\|
      \big(1-\chi_{\ell,\alpha,\varepsilon}^{(i)}\big) \,
      \Big( \omega_1^{(i)}(0) + x_i \frac{d\omega_1^{(i)}}{dx_i}(0) - \omega_1^{(i)} \Big)
     \bigg\|_{H^1(\Omega_\varepsilon^{(i)})}
$$
$$
   + \varepsilon^2 \sum_{i=1}^3
           \left\|
       \big(1-\chi_{\ell,\alpha,\varepsilon}^{(i)}\big) \,
       \big(u_2^{(i)}(0,\cdot) - u_2^{(i)} \big)
      \right\|_{H^1(\Omega_\varepsilon^{(i)})}
+  \varepsilon^2 \sum_{i=1}^3 \left\|
       \big(1-\chi_{\ell,\alpha,\varepsilon}^{(i)}\big) \,
       \big(\omega_2^{(i)}(0) - \omega_2^{(i)}\big)
      \right\|_{H^1(\Omega_\varepsilon^{(i)})}
$$
$$
   + \sum_{i=1}^3
     \left(
      \varepsilon
      \left\| \,
       \omega_1^{(i)}
      \right\|_{H^1(\Omega_\varepsilon^{(i)})}
   +  \varepsilon^2
      \left\| \,
       u_2^{(i)} + \omega_2^{(i)}
      \right\|_{H^1(\Omega_\varepsilon^{(i)})}
     \right)
$$
$$
   + \sum_{i=1}^3
     \bigg(
      \varepsilon
      \left\| \,
       \big(1-\chi_{\ell,\alpha,\varepsilon}^{(i)}\big) \,  ( N_1 - G_1 )
      \right\|_{H^1(\Omega_\varepsilon^{(i)})}
   +  \varepsilon^2
      \left\|
       \big(1-\chi_{\ell,\alpha,\varepsilon}^{(i)}\big) \, ( N_2 - G_2 )
      \right\|_{H^1(\Omega_\varepsilon^{(i)})}
     \bigg)
$$
$$
   + \varepsilon^\frac32 \left\| N_1 \right\|_{H^1(\Xi^{(0)})}
   + \varepsilon^\frac52 \left\| N_2 \right\|_{H^1(\Xi^{(0)})}
   + \varepsilon^2 \sum_{i=1}^3
     \left\| \, \chi_\delta^{(i)} \Pi_2^{(i)} \right\|_{H^1(\Omega_\varepsilon^{(i)})}
 \le \widetilde{C}_0 \, \varepsilon^{1+\frac{\alpha}{2}}.
$$
To prove the second estimate in (\ref{t5}), we need to calculate the $L^2$-norm of terms in the right-hand side of the previous inequality.

The inequality (\ref{t6}) can be similarly obtained  from the estimate (\ref{t0}) at $m=4$.

Again with the help of estimate (\ref{t0}) at $m=4,$ we deduce
\begin{gather*}
      \left\| \,
       u_\varepsilon - \omega_0^{(i)}
      \right\|_{H^1(\Omega_{\varepsilon,\alpha}^{(i)})}
 \leq \left\| \,
       u_\varepsilon-U^{(4)}_{\varepsilon}
      \right\|_{H^1(\Omega_\varepsilon)}
    + \varepsilon
      \left\| \,
       \omega_{1}^{(i)}
      \right\|_{H^1(\Omega_{\varepsilon,\alpha}^{(i)})}
\\
    + \sum_{k=2}^4 \varepsilon^{k}
      \left\| \,
       u_{k}^{(i)} + \omega_{k}^{(i)} + \chi_\delta^{(i)} \Pi_{k}^{(i)}
      \right\|_{H^1(\Omega_{\varepsilon,\alpha}^{(i)})}
 \leq \widetilde{C}_2 \, \varepsilon^{2},
\end{gather*}
whence we get (\ref{t7}).

From inequality
$$
      \left\| \,
       u_\varepsilon - \omega_2^{(i)}(0) - \varepsilon \, N_1
      \right\|_{H^1(\Omega^{(0)}_{\varepsilon, \ell})}
 \leq \left\| \,
       u_\varepsilon - U^{(4)}_{\varepsilon}
      \right\|_{H^1(\Omega_\varepsilon)}
    + \sum_{k=2}^4 \varepsilon^k
      \| N_k \|_{H^1(\Omega^{(0)}_{\varepsilon, \ell})}
 \leq \widetilde{C}_4 \, \varepsilon^{\frac52}
$$
it follows more better energetic estimate (\ref{t-joint0}) in a neighbourhood of the aneurysm $\Omega^{(0)}_\varepsilon.$
\end{proof}

Using the Cauchy-Buniakovskii-Schwarz inequality and  the continuously embedding of
the space $H^1(I_{\varepsilon, \alpha}^{(i)})$  in $C\big(\overline{I_{\varepsilon, \alpha}^{(i)}}\big),$
it follows from (\ref{t7}) the following corollary.

\begin{corollary}\label{corollary2}
If $h_i(x_i) \equiv h_i \equiv const, \, (i=1,2,3),$ then
\begin{equation}\label{t9}
 \| \, E^{(i)}_\varepsilon(u_\varepsilon) - \omega_0^{(i)}  \|_{H^1(I_{\varepsilon, \alpha}^{(i)})} \leq \widetilde{C}_2 \,
 \varepsilon,
\end{equation}
\begin{equation}\label{t10}
\max_{x_i\in \overline{I_{\varepsilon, \alpha}^{(i)}}} \left| \, E^{(i)}_\varepsilon(u_\varepsilon)(x_i) - \omega_0^{(i)}(x_i)
\right|  \leq \widetilde{C}_3 \,  \varepsilon, \ \ i=1,2,3,
\end{equation}
where
$$
\big(E^{(i)}_\varepsilon u_\varepsilon\big)(x_i)
 = \frac{1}{\pi \varepsilon^2\, h_i^2}
   \int_{\Upsilon^{(i)}_\varepsilon(0)}
   u_\varepsilon(x)\, d\overline{x}_i,
\quad i=1,2,3.
$$
\end{corollary}

\begin{corollary}\label{corollary3}
If to the assumptions $h_i(x_i) \equiv h_i \equiv const, \, i=1,2,3,$ the function $\varphi_\varepsilon \equiv 0$ and $f=f(x_i), \ x\in\Omega_\varepsilon^{(i)} \, i=1,2,3,$ then  the asymptotic expansion for the solution $u_\varepsilon$ has the following more simple form:
\begin{equation}\label{asymp_expansion1}
     \sum\limits_{k=0}^{+\infty} \varepsilon^{k}
     \Bigg(
      \sum_{i=1}^3 \chi_\ell^{(i)} \left(\frac{x_i}{\varepsilon^\alpha}\right) \omega_{k}^{(i)}(x_i)
   +  \left( 1 - \sum_{i=1}^3 \chi_\ell^{(i)} \left(\frac{x_i}{\varepsilon^\alpha}\right) \right)
      {N}_{k} \left(\frac{x}{\varepsilon}\right)
     \Bigg),
\quad x\in\Omega_\varepsilon,
\end{equation}
and the asymptotic estimates are improved:
\begin{equation}\label{t0+}
      \| \, u_\varepsilon - U_\varepsilon^{(m)}\|_{H^1(\Omega_\varepsilon)}
 \leq {C}_m \ \varepsilon^{\alpha (m\,-\frac{1}{2}) + 1 }.
\end{equation}
\begin{equation}\label{t8+}
\|u_\varepsilon - \omega^{(i)}_0 - \varepsilon \omega^{(i)}_1 \|_{H^1(\Omega_{\varepsilon,\alpha}^{(i)})}
 \leq {C}_1 \, \varepsilon^{3}, \ \ i=1,2,3;
\end{equation}
\begin{equation}\label{t9+}
      \|
       E^{(i)}_\varepsilon u_\varepsilon - \omega^{(i)}_0 - \varepsilon \omega^{(i)}_1
      \|_{H^1(I_{\varepsilon, \alpha}^{(i)})}
 \leq {C}_2 \, \varepsilon^{2}, \ \ i=1,2,3;
\end{equation}
\begin{equation}\label{t10+}
\max_{x\in \overline{I_{\varepsilon, \alpha}^{(i)}}}
\left| \big(E^{(i)}_\varepsilon u_\varepsilon\big)(x_i) - \omega^{(i)}_0(x_i) - \varepsilon \omega^{(i)}_1(x_i) \right|
 \leq {C}_3 \,  \varepsilon^{2}, \ \ i=1,2,3.
\end{equation}
\end{corollary}

\section{Conclusions}\label{Conclusions}

{\bf 1.}
An important problem of existing multi-scale methods is their stability and accuracy. The proof of the
error estimate between the constructed approximation and the exact solution is a general principle that has been
applied to the analysis of the efficiency of a multi-scale method. In our paper, we have constructed and justified the asymptotic
expansion for the solution to problem (\ref{probl}) and proved the corresponding estimates.
It should be noted here that we do not assume any orthogonality conditions for the right-hand sides in the equation and in the Neumann boundary conditions.

\smallskip

The results showed the possibility to replace the complex boundary-value problem (\ref{probl}) with
the corresponding $1$- dimensional boundary-value problem (\ref{main}) in the graph $I=\cup_{i=1}^3 I_i$  with sufficient accuracy measured by the parameter $\varepsilon$ characterizing the thickness and the local geometrical irregularity. In this regard, the uniform pointwise estimates (\ref{t10}) and (\ref{t10+}), that are very important for applied problems, also confirm this conclusion.

\medskip
\noindent
{\bf 2.}
To construct the asymptotic expansion in the whole domain, we have used the  method of matching asymptotic expansions with special cut-off functions. It is the natural approach for approximations of solutions to boundary-value problems in perturbed domains. In  comparison to the method
 of the partial asymptotic domain decomposition \cite{Pan-decom-1998}, this method gives better estimate
even for the first terms $\{\omega_0^{(i)}\}_{i=1}^3$ in the $L^2$-norm (compare (\ref{P}) and the second estimate in (\ref{t5}))
without any additional assumptions for the right-hand sides.

\medskip
\noindent
{\bf 3.}
The energetic estimate (\ref{t5}) partly confirms the first formal result of \cite{Gaudiello-Kolpakov} (see p.~296)  that the local geometric irregularity of  the analyzed structure  does not significantly affect on  the global-level properties of the framework, which are described by the limit problem (\ref{main}) and its solution $\{\omega^{(i)}_0\}_{i=1}^3$ (the first terms of the asymptotics).

Therefore, convergence results, which were obtained for second-order problems in a thin T-like shaped domain (see \cite{GGLM,GS} and references therein) cannot show the influence of the aneurysm.
But thanks  to estimates (\ref{t6}) and (\ref{t8+}) -- (\ref{t10+})
it became possible now to identify the impact of
the geometric irregularity and material characteristics of the aneurysm on the global level through  the second terms $\{\omega^{(i)}_1\}_{i=1}^3$ of the regular asymptotics (\ref{regul}). They  depend on the constants $d_1^*,$  $\delta_{1}^{(2)}$ and $\delta_{1}^{(3)}$ that take into account all those factors (see (\ref{delta_1}) and (\ref{d_1^*})).
This conclusion does not coincide with the second main result of \cite{Gaudiello-Kolpakov} (see p. 296) that
``{\it the joints of normal type manifest themselves on the local level only}''.

\smallskip

 In addition, in \cite{Gaudiello-Kolpakov} the authors stated that the main idea of their approach ``{\it
is to use a local perturbation corrector of the form $\varepsilon N({\bf x}/\varepsilon) \frac{d u_0}{dx_1}$ with the condition that the function $N({\bf y})$
is localized near the joint} '', i.e., $N({\bf y}) \to 0$ as $|{\bf y}| \to +\infty,$ and the main assumption of this approach is that $\nabla_{y} N \in L_1(Q_{\infty}).$

As shown the coefficients $\{N_k\}$ of the inner asymptotics (\ref{junc}) behave as polynomials at infinity and do not decrease exponentially (see (\ref{inner_asympt})).  Therefore, they influence directly the terms of the regular asymptotics beginning with the second terms.
Thus, the main assumption made in \cite{Gaudiello-Kolpakov} is not correct.

\medskip
\noindent
{\bf 4.}
From the first estimate in (\ref{t5}) it follows that the gradient $\nabla u_\varepsilon$ is equivalent to $\{\frac{d\omega^{(i)}_0}{dx_i}\}_{i=1}^3$ in the $L^2$-norm over whole junction $\Omega_\varepsilon$ as $\varepsilon \to 0.$
Obviously, this estimate is not informative in the neighbourhood $\Omega^{(0)}_{\varepsilon, l}$ of the aneurysm $\Omega^{(0)}_{\varepsilon}.$

Thanks to estimates (\ref{t6}) and (\ref{t-joint0}), we get  the approximation
of the gradient (flux) of the solution both in the curvilinear cylinders~$\Omega^{(i)}_{\varepsilon, \alpha},$ $i=1,2, 3$:
$$
\nabla u_\varepsilon(x)  \sim \frac{d\omega^{(i)}_0}{dx_i}(x_i) + \varepsilon \, \frac{d\omega^{(i)}_1}{dx_i}(x_i) \quad \text{as}\quad \varepsilon \to 0
$$
and in the neighbourhood $\Omega^{(0)}_{\varepsilon, l}$ of the aneurysm:
$$
\nabla u_\varepsilon(x) \sim \nabla_{\xi}\big({N}_{1}(\xi)\big)\Big|_{\xi=\frac{x}{\varepsilon}} \quad \text{as}\quad \varepsilon \to 0.
$$
Also using estimates (\ref{t0}), we can obtain better approximations for the solution and its gradient with preset accuracy
${\cal O}(\varepsilon^{\alpha (m\,-\frac{1}{2}) + \frac{1}{2} }),$ $\forall m\in \Bbb N.$

\medskip
\noindent
{\bf 5.}
We regard that the estimates (\ref{t9}) and (\ref{t10}) can be proved without the assumptions that $h_i(x_i) \equiv h_i \equiv const, \, (i=1,2,3).$
For this we should apply the following second energy inequality  in~$\Omega_\varepsilon:$
$$
\|u_\varepsilon\|_{H^2(\Omega_\varepsilon)} \le M \Big(\|u_\varepsilon\|_{H^1(\Omega_\varepsilon)} + \|f\|_{L^2(\Omega_\varepsilon)}
+ \sum_{i=1}^3 \|\varphi_\varepsilon\|_{H^{\frac12}(\Gamma^{(i)}_\varepsilon)}\Big).
$$
But the question how the constant $M$  depends on the parameter $\varepsilon$ remains open.
The answer will allow to get  better estimates both in the norms of energy spaces and in the uniform metric.


\end{document}